\documentclass[reqno,11pt]{amsart}

\usepackage{amsmath,amssymb,amsthm,mathtools}
\usepackage{geometry}
\usepackage{enumitem}
\usepackage{microtype}
\usepackage{xcolor}
\usepackage{hyperref}
\usepackage{aliascnt}
\usepackage[nameinlink,capitalise,noabbrev]{cleveref}
\usepackage{mathrsfs}

\allowdisplaybreaks

\hypersetup{
  colorlinks=true,
  linkcolor=blue!55!black,
  citecolor=blue!55!black,
  urlcolor=blue!55!black
}

\numberwithin{equation}{section}

\newtheorem{theorem}{Theorem}[section]
\newaliascnt{proposition}{theorem}
\newtheorem{proposition}[proposition]{Proposition}
\aliascntresetthe{proposition}
\newaliascnt{lemma}{theorem}
\newtheorem{lemma}[lemma]{Lemma}
\aliascntresetthe{lemma}
\newaliascnt{corollary}{theorem}
\newtheorem{corollary}[corollary]{Corollary}
\aliascntresetthe{corollary}
\newaliascnt{claim}{theorem}

\aliascntresetthe{claim}
\theoremstyle{definition}
\newaliascnt{definition}{theorem}

\aliascntresetthe{definition}
\theoremstyle{remark}
\newaliascnt{remark}{theorem}
\newtheorem{remark}[remark]{Remark}
\aliascntresetthe{remark}

\crefname{theorem}{Theorem}{Theorems}
\Crefname{theorem}{Theorem}{Theorems}
\crefname{proposition}{Proposition}{Propositions}
\Crefname{proposition}{Proposition}{Propositions}
\crefname{lemma}{Lemma}{Lemmas}
\Crefname{lemma}{Lemma}{Lemmas}
\crefname{corollary}{Corollary}{Corollaries}
\Crefname{corollary}{Corollary}{Corollaries}
\crefname{claim}{Claim}{Claims}
\Crefname{claim}{Claim}{Claims}
\crefname{definition}{Definition}{Definitions}
\Crefname{definition}{Definition}{Definitions}
\crefname{remark}{Remark}{Remarks}
\Crefname{remark}{Remark}{Remarks}

\newcommand{\R}{\mathbb{R}}
\newcommand{\Sn}{\mathbb{S}}
\newcommand{\cE}{\mathcal{E}}
\newcommand{\cG}{\mathcal{G}}
\newcommand{\cK}{\mathcal{K}}
\newcommand{\cT}{\mathcal{T}}
\newcommand{\cO}{\mathcal{O}}
\newcommand{\eps}{\varepsilon}
\newcommand{\tr}{\operatorname{tr}}
\newcommand{\In}{\operatorname{In}}
\newcommand{\dist}{\operatorname{dist}}
\newcommand{\diag}{\operatorname{diag}}
\newcommand{\PV}{\operatorname{PV}}
\newcommand{\Span}{\operatorname{span}}
\newcommand{\Id}{\mathrm{Id}}
\newcommand{\dd}{\,\mathrm{d}}

\title{Strict Concavity of the Torsion Function for the Restricted Half-Laplacian in Bounded Convex Domains}
\author{Jiahuan Li}
\address{School of Mathematical Sciences, University of Science and Technology of China, Hefei 230026, Anhui Province, China.}
\email{jiahuan@mail.ustc.edu.cn}

\author{Shujun Shi}
\address{School of Mathematical Sciences, Harbin Normal University, Harbin 150025, Heilongjiang Province, China.}
\email{shjshi@hrbnu.edu.cn}
\date{}

\begin{document}
\pagestyle{plain}

\begin{abstract}
Let $D\subset\R^n$, $n\ge2$, be a bounded convex domain, and let $u_D$ be the torsion function for the restricted half-Laplacian. We prove that $D^2u_D$ is negative definite at every point of $D$. The argument is based on the reflected harmonic extension in a slit domain. Quantitative Schauder estimates in slit domains yield parameter-uniform estimates for the first and second derivatives of the edge remainder; a Schur-complement calculation then determines the inertia of the extended Hessian near the slit edge. Superharmonicity of the logarithmic Hessian determinant and the Gleason--Wolff zero-set theorem exclude interior degeneracy. A method of continuity starting from the unit ball proves the result for smooth uniformly convex domains, and an exhaustion argument treats arbitrary bounded convex domains.
\end{abstract}

\medskip

\maketitle

\section{Introduction}

For a bounded open set $D\subset\R^n$, we consider the restricted half-Laplacian torsion problem
\begin{equation}\label{eq:torsion}
 \begin{cases}
 (-\Delta)^{1/2}u=1&\text{in }D,\\
 u=0&\text{in }D^c,
 \end{cases}
 \qquad
 (-\Delta)^{1/2}f(x)=c_n\,\PV\!\int_{\R^n}
 \frac{f(x)-f(z)}{|x-z|^{n+1}}\,\dd z,
\end{equation}
where $c_n=\Gamma((n+1)/2)/\pi^{(n+1)/2}$. For a symmetric matrix, the inequalities $A<0$ and $A>0$ mean negative and positive definiteness, respectively. We determine the pointwise sign of $D^2u_D$ when $D$ is convex.

The solution of \eqref{eq:torsion} is the torsion function for the Dirichlet realization of the restricted fractional Laplacian. Equivalently, up to the normalization in \eqref{eq:torsion}, $u_D(x)$ is the expected lifetime in $D$ of the isotropic Cauchy process started at $x$; see \cite{Getoor1961,BogdanEtAl2009}. The torsion function therefore has both an analytic interpretation through the nonlocal Dirichlet problem and a probabilistic interpretation through the Cauchy process. The exterior condition in \eqref{eq:torsion} is imposed on all of $D^c$, rather than only on $\partial D$.

For comparison, let $v$ denote the classical torsion function, which satisfies
\[
 -\Delta v=1\quad\text{in }D,
 \qquad v=0\quad\text{on }\partial D.
\]
Then $v^{1/2}$ is concave in every bounded convex domain. Makar-Limanov proved the planar result \cite{MakarLimanov1971}; Borell and Kennington established the result in arbitrary dimension by convexity maximum principles \cite{Borell1985,Kennington1985}, following related ideas in \cite{Korevaar1983}. Equivalently, the classical torsion function is $1/2$-concave. Related results include log-concavity of the first Dirichlet eigenfunction \cite{BrascampLieb1976}, convexity for semilinear elliptic equations \cite{CaffarelliFriedman1985}, and the constant-rank theorem of Korevaar--Lewis \cite{KorevaarLewis1987}. Constant-rank and microscopic convexity principles for fully nonlinear elliptic equations were subsequently developed in \cite{CaffarelliGuanMa2007,BianGuan2009}.

Another class of arguments derives geometric information from subharmonic or superharmonic quantities constructed from the Hessian, the gradient, or the second fundamental forms of level sets. Results for harmonic and $p$-harmonic level sets appear in \cite{MaOuZhang2010,MaZhang2021}, and Hessian estimates for Green functions and related elliptic problems appear in \cite{MaShiYe2012,Shi2015}. Jia--Ma--Shi proved quantitative convexity estimates through superharmonic quantities for the classical torsion problem and also gave strict $1/2$-concavity   in \cite{JiaMaShi2023}, with related estimates for the Green function and the first Dirichlet eigenfunction in \cite{JiaMaShiGreen2023}. For harmonic gradient mappings, Gleason--Wolff \cite{GleasonWolff1991} proved a higher-dimensional Lewy theorem under a bound on the Morse index of the Hessian and also gave superharmonicity of the logarithmic Hessian determinant needed in the present argument.

Fewer pointwise concavity results are available for nonlocal operators. The harmonic-extension realization of $(-\Delta)^{1/2}$, and more generally the degenerate extension for fractional powers, is fundamental \cite{CaffarelliSilvestre2007}; boundary regularity for the restricted fractional Dirichlet problem is developed in \cite{RosOtonSerra2014,Grubb2015}. Related results for fractional ground states include mid-concavity on intervals and rectangles \cite{BanuelosKulczyckiMendez2006}, as well as superharmonicity results for particular stability indices and geometries \cite{BanuelosDeBlassie2015,AbatangeloJarohs2024}.

The result most directly preceding the present work is Kulczycki's theorem \cite{Kulczycki2017}: for $n=2$, the solution of \eqref{eq:torsion} is concave in every bounded convex domain. His proof uses the reflected harmonic extension and a method of continuity based on Lewy's three-dimensional Hessian theorem. Kulczycki also showed that concavity may fail for torsion functions of $(-\Delta)^{\alpha/2}$, $1<\alpha<2$, on sufficiently narrow bounded cones. The case $\alpha=1$ is therefore an endpoint problem requiring a separate argument. He conjectured that the conclusion for $\alpha=1$ holds in every dimension. We prove this conjecture and obtain the stronger pointwise Hessian inequality below.

\begin{theorem}\label{thm:main}
Let $n\ge2$, and let $D\subset\R^n$ be a bounded convex domain. Then the unique weak solution $u_D$ of \eqref{eq:torsion} belongs to $C^\infty(D)$ and
\begin{equation}\label{eq:main-strict}
 D^2u_D(x)<0\qquad(x\in D).
\end{equation}
Consequently, the restriction of $u_D$ to every line segment with distinct endpoints in $D$ is strictly concave.
\end{theorem}

The conclusion $D^2u_D<0$ is stronger than strict concavity understood only as a chord inequality: it excludes every nontrivial null vector of the Hessian at every interior point. It is invariant under similarities. More precisely, if $\widetilde D=aD+b$ with $a>0$, then
\begin{equation}\label{eq:similarity-scaling}
 u_{\widetilde D}(x)=a\,u_D\!\left(\frac{x-b}{a}\right),
 \qquad
 D^2u_{\widetilde D}(x)=a^{-1}D^2u_D\!\left(\frac{x-b}{a}\right).
\end{equation}

We outline the argument. Let $U_D$ denote the Poisson extension of $u_D$ to $\R^{n+1}_+$ and define
\[
 W_D(x,y)=U_D(x,|y|)+|y|.
\]
The Neumann condition on $D\times\{0\}$ implies that $W_D$ extends harmonically across this set; hence $W_D$ is harmonic in the slit domain
\[
 \mathcal G_D=\R^{n+1}\setminus(D^c\times\{0\}).
\]
It is enough to prove that $D^2W_D$ has one positive and $n$ negative eigenvalues in $\mathcal G_D$ and that $(W_D)_{yy}(x,0)>0$ for $x\in D$.

For a smooth uniformly convex domain, choose a smooth path $\{D_\tau\}_{\tau\in[0,1]}$ from the unit ball to $D$ and introduce the harmonic perturbation
\[
 V_{D_\tau,\eps}=W_{D_\tau}+\eps Q_n,
 \qquad Q_n(x,y)=\tfrac12(ny^2-|x|^2).
\]
Quantitative Schauder estimates of De Silva--Savin yield a slit-edge expansion with differentiated remainder estimates uniform along the domain path. A Schur-complement calculation then proves that $D^2V_{D_\tau,\eps}$ has inertia $(1,n)$ in a tubular neighborhood of $\partial D_\tau\times\{0\}$. Taylor expansion on $D_\tau^c\times\{0\}$ away from the edge and the far-field expansion of the Poisson kernel give the same inertia near the other boundary components of a truncated slit domain.

On the complementary interior subdomain, we apply the logarithmic determinant inequality to $q=-V_{D_\tau,\eps}$. If $n_-(D^2q)=1$, then
\[
 \Delta\log|\det D^2q|\le0.
\]
The minimum principle extends a positive boundary lower bound for $|\det D^2q|$ to the interior. The explicit unit-ball computation is the base point of the continuation argument, which yields the inertia statement for every $\tau\in[0,1]$. Passing to the limit as $\eps\downarrow0$ gives $n_+(D^2W_D)\le1$.

The limiting Hessian may be singular. The Gleason--Wolff zero-set theorem, together with a nondegenerate point obtained from the far-field expansion, implies that $\det D^2W_D$ is nowhere zero in $\mathcal G_D$. Hence $D^2W_D$ has inertia $(1,n)$ throughout the slit domain. Finally, the maximum principle and the Hopf boundary lemma applied to $Z=(W_D)_y$ give $(W_D)_{yy}(x,0)>0$ on $D\times\{0\}$. The block decomposition of $D^2W_D(x,0)$ then yields $D_x^2u_D(x)<0$. Approximation by smooth uniformly convex subdomains gives the general case.

The parameter-uniform edge estimate is proved in \cref{prop:uniform-edge} from the quantitative Schauder theory of De Silva--Savin for slit domains \cite{DeSilvaSavin2015}. We estimate exactly the three Hessian blocks that enter the Schur complement. Derivatives in the extension variable are obtained from scale-invariant interior and flat Dirichlet boundary estimates. The second external result used in the proof is the Gleason--Wolff zero-set theorem \cite{GleasonWolff1991}; the logarithmic determinant inequality is given in \cref{thm:det-min}.

The paper is organized as follows. Section~2 establishes the variational formulation and the reflected harmonic extension. Section~3 proves the logarithmic Hessian determinant inequality and records the Gleason--Wolff zero-set theorem. Section~4 proves the parameter-uniform slit-edge expansion and derives the Hessian inertia near the slit edge. Section~5 treats the exterior Dirichlet boundary away from the edge and the far-field region. Section~6 contains the explicit unit-ball computation, and Section~7 proves stability under smooth domain perturbations. Sections~8--10 apply the method of continuity, prove strict concavity for smooth uniformly convex domains, and pass to general convex domains. The appendix contains the determinant computation for the unit ball.

\section{Variational formulation and reflected harmonic extension}\label{sec:extension}

\subsection{Weak solutions and comparison}

Let $H^{1/2}_0(D)$ denote the closure of $C_c^\infty(D)$ in $H^{1/2}(\R^n)$. At the critical exponent $1/2$, this definition is stronger than requiring an $H^{1/2}(\R^n)$ function to vanish almost everywhere in $D^c$. The part-form characterization for the regular fractional Dirichlet form gives
\[
 H^{1/2}_0(D)
 =\{v\in H^{1/2}(\R^n):\widetilde v=0
 \text{ quasi-everywhere on }D^c\},
\]
where $\widetilde v$ denotes the quasi-continuous representative; see \cite[Section~4.4]{FukushimaOshimaTakeda2011}. Define the symmetric bilinear form
\begin{equation}\label{eq:energy}
 \cE(v,w)
 =\frac{c_n}{2}\iint_{\R^n\times\R^n}
 \frac{(v(x)-v(z))(w(x)-w(z))}{|x-z|^{n+1}}\,\dd x\,\dd z.
\end{equation}
The weak solution $u_D$ of \eqref{eq:torsion} is characterized by
\begin{equation}\label{eq:weak}
 \cE(u_D,\zeta)=\int_D\zeta\,\dd x
 \qquad(\zeta\in H^{1/2}_0(D)).
\end{equation}

\begin{lemma}\label{lem:variational}
For every bounded open set $D$, problem \eqref{eq:weak} has a unique solution $u_D\in H^{1/2}_0(D)$. Moreover $u_D\ge0$, and
\begin{equation}\label{eq:domain-monotonicity}
 D_1\subset D_2
 \quad\Longrightarrow\quad
 0\le u_{D_1}\le u_{D_2}\quad\text{a.e. in }\R^n.
\end{equation}
\end{lemma}

\begin{proof}
We separate the three assertions.

\emph{Existence and uniqueness.}
On a bounded set, the fractional Poincar\'e inequality controls the full $H^{1/2}$ norm by the Gagliardo seminorm appearing in $\cE$. Thus $\cE$ is coercive on $H^{1/2}_0(D)$. The functional
\[
 \zeta\longmapsto\int_D\zeta\,\dd x
\]
is continuous on the same space by the fractional Sobolev inequality and H\"older's inequality. The Lax--Milgram theorem therefore gives a unique $u_D$ satisfying \eqref{eq:weak}.

\emph{Positivity.}
For $t\in\R$, set $t^+=\max\{t,0\}$ and $t^-=\max\{-t,0\}$. The normal-contraction property of the fractional Dirichlet form implies that $u_D^-\in H^{1/2}_0(D)$. For real numbers $a,b$,
\[
 (a-b)(a^- -b^-)\le-|a^- -b^-|^2.
\]
We may therefore use $u_D^-$ as a test function in \eqref{eq:weak} and obtain
\[
 0\le\int_Du_D^-\,\dd x
 =\cE(u_D,u_D^-)
 \le-\cE(u_D^-,u_D^-)\le0.
\]
It follows that $\cE(u_D^-,u_D^-)=0$. Coercivity gives $u_D^-=0$, hence $u_D\ge0$ almost everywhere.

\emph{Domain monotonicity.}
Assume $D_1\subset D_2$ and set
\[
 w=(u_{D_1}-u_{D_2})^+.
\]
Normal contractions preserve the form domain, so $w\in H^{1/2}_0(D_2)$. Let tildes denote quasi-continuous representatives. On $D_1^c$ one has $\widetilde u_{D_1}=0$ quasi-everywhere and $\widetilde u_{D_2}\ge0$ quasi-everywhere; consequently, $\widetilde w=0$ quasi-everywhere on $D_1^c$. The characterization stated above therefore gives $w\in H^{1/2}_0(D_1)$. Thus $w$ is admissible in both weak formulations. Since $w=0$ almost everywhere in $D_2\setminus D_1$, subtracting the two identities yields
\[
 \cE(u_{D_1}-u_{D_2},w)=0.
\]
The elementary inequality
\[
 (a-b)(a^+-b^+)\ge|a^+-b^+|^2
\]
then implies $0\ge\cE(w,w)$. Thus $w=0$, which is exactly $u_{D_1}\le u_{D_2}$. Together with positivity, this proves \eqref{eq:domain-monotonicity}.
\end{proof}

\subsection{Harmonic extension and reflected slit-domain formulation}

Put $N=n+1$, write $X=(x,y)\in\R^n\times\R$, denote the last coordinate vector by
\[
 e_y=(0,\ldots,0,1)\in\R^N,
\]
and set
\[
 \R^N_+:=\{(x,y)\in\R^n\times\R:y>0\}.
\]
We denote by $B_R^k(X_0)$ the Euclidean ball of radius $R$ in $\R^k$; the center and the superscript are omitted when they are clear from context. For positive quantities $A$ and $B$, the notation $A\simeq B$ means that $cB\le A\le CB$ with constants independent of the parameters under consideration. For $y>0$, define
\begin{equation}\label{eq:poisson-extension}
 U_D(x,y)
 =c_n\int_D
 \frac{y\,u_D(z)}{(|x-z|^2+y^2)^{N/2}}\,\dd z.
\end{equation}
The kernel in \eqref{eq:poisson-extension} is the Poisson kernel for the upper half-space. Taking the Fourier transform only in the $x$ variables gives
\[
 \widehat U_D(\xi,y)=e^{-y|\xi|}\widehat u_D(\xi).
\]
Consequently, $U_D$ is harmonic in $\R^{N}_+$, and the Dirichlet-to-Neumann operator $-\partial_y$ has Fourier symbol $|\xi|$. Equation \eqref{eq:weak} is therefore equivalent to the mixed boundary conditions
\begin{equation}\label{eq:mixed-boundary}
 \Delta_{x,y}U_D=0\quad(y>0),
 \qquad
 (U_D)_y=-1\quad\text{on }D\times\{0\},
 \qquad
 U_D=0\quad\text{on }D^c\times\{0\}.
\end{equation}
The boundary conditions are understood in the variational trace sense and hold classically on compact subsets of the relative interiors of $D\times\{0\}$ and $D^c\times\{0\}$.
Define the even, affine-corrected reflection
\begin{equation}\label{eq:W}
 W_D(x,y)=U_D(x,|y|)+|y|.
\end{equation}
For $y>0$ and $x\in D$, one has $\partial_y(U_D+y)=0$ at $y=0$. Hence the even reflection in \eqref{eq:W} is weakly harmonic across $D\times\{0\}$. Interior elliptic regularity then gives
\begin{equation}\label{eq:slit-domain}
 W_D\in C^\infty(\cG_D),
 \qquad
 \Delta W_D=0\quad\text{in }\cG_D,
 \qquad
 \cG_D:=\R^{N}\setminus(D^c\times\{0\}).
\end{equation}
For $y>0$, the affine correction does not affect the Hessian:
\begin{equation}\label{eq:hessian-same}
 D^2W_D=D^2U_D.
\end{equation}
At every $x\in D$, the function $W_D$ is smooth in a full neighborhood of $(x,0)$, it is even in $y$, and
\begin{equation}\label{eq:bottom-block}
 D^2W_D(x,0)
 =\begin{pmatrix}
 D_x^2u_D(x)&0\\
 0&(W_D)_{yy}(x,0)
 \end{pmatrix}.
\end{equation}

For a real symmetric matrix $A$, let $n_+(A)$, $n_-(A)$, and $n_0(A)$ denote the numbers of positive, negative, and zero eigenvalues, counted with multiplicity. If $A$ is nonsingular, write
\[
 \In A=(n_+(A),n_-(A)),
\]
and call this pair the \emph{inertia} of $A$. We shall prove
\begin{equation}\label{eq:target-inertia}
 \In D^2W_D=(1,n).
\end{equation}
In dimension $N=3$, the trace-free condition and the sign of the determinant determine this inertia. In higher dimensions they do not; one must control the number of positive eigenvalues throughout the argument.

Introduce the harmonic quadratic polynomial
\begin{equation}\label{eq:Q}
 Q_n(x,y)=\frac12(ny^2-|x|^2),
 \qquad
 D^2Q_n=\diag(-I_n,n),
\end{equation}
and set
\begin{equation}\label{eq:V}
 V_{D,\eps}=W_D+\eps Q_n,
 \qquad \eps>0.
\end{equation}
The $x$-block of $D^2Q_n$ is $-I_n$. The method of continuity is applied to $V_{D,\eps}$ for fixed $\eps>0$; the unperturbed extension is recovered by letting $\eps\downarrow0$.

\section{the Gleason--Wolff logarithmic Hessian determinant inequality and  zero-set theorem}\label{sec:det-min}

The superharmonicity of the logarithmic Hessian determinant in this section is due to Gleason–Wolff \cite{GleasonWolff1991}. For convenience, we provide  a self-contained proof through a trace-free symmetric three-tensor inequality. We then state the Gleason--Wolff zero-set theorem.

\subsection{An algebraic inequality for trace-free symmetric three-tensors}

The following algebraic lemma is the pointwise estimate used in the proof of \cref{thm:det-min}.

\begin{lemma}\label{lem:tensor}
Let $m\ge1$, let $\lambda_1,\dots,\lambda_m>0$, and put
\[
 \Lambda=\sum_{\alpha=1}^m\lambda_\alpha.
\]
Set $\lambda_0=-\Lambda$. Let $T=(T_{ijk})_{0\le i,j,k\le m}$ be a completely symmetric three-tensor satisfying
\begin{equation}\label{eq:trace-free-tensor}
 \sum_{i=0}^m T_{iik}=0
 \qquad(0\le k\le m).
\end{equation}
Then
\begin{equation}\label{eq:tensor-ineq}
 \sum_{i,j,k=0}^m\frac{T_{ijk}^2}{\lambda_i\lambda_j}\ge0.
\end{equation}
\end{lemma}

\begin{proof}
The proof is an explicit decomposition of the quadratic form. We normalize the positive eigenvalues by writing
\[
 \mu_\alpha=\frac{\lambda_\alpha}{\Lambda},
 \qquad
 \mu_\alpha>0,
 \qquad
 \sum_{\alpha=1}^m\mu_\alpha=1.
\]
Because $T$ is symmetric, its components fall into disjoint groups according to how many indices are equal to $0$. We verify nonnegativity group by group.

\smallskip
\noindent\emph{Group 1: the components $T_{000}$ and $T_{0\alpha\alpha}$.}
Put
\[
 y_\alpha=T_{0\alpha\alpha},
 \qquad
 z_\alpha=\frac{y_\alpha}{\lambda_\alpha}.
\]
The constraint \eqref{eq:trace-free-tensor} for $k=0$ gives $T_{000}=-\sum_\alpha y_\alpha$. The contribution of these components to the left-hand side of \eqref{eq:tensor-ineq} is
\begin{equation}\label{eq:Q0}
 \mathscr Q_0(z)
 =\left(\sum_\alpha\mu_\alpha z_\alpha\right)^2
 +\sum_\alpha(1-2\mu_\alpha)z_\alpha^2.
\end{equation}
If every $\mu_\alpha\le1/2$, then each term in \eqref{eq:Q0} is nonnegative. Suppose that $\mu_1>1/2$; it is the only such index. Put
\[
 \mathsf D=\diag(1-2\mu_2,\dots,1-2\mu_m),
 \qquad
 \Sigma_\mu=\sum_{\alpha=2}^m\frac{\mu_\alpha^2}{1-2\mu_\alpha}.
\]
The Schur complement of the positive block $\mathsf D+(\mu_\alpha\mu_\beta)_{\alpha,\beta\ge2}$ in the quadratic form \eqref{eq:Q0} is
\begin{equation}\label{eq:Q0-schur}
 1-2\mu_1+\frac{\mu_1^2}{1+\Sigma_\mu}.
\end{equation}
Let $r=1-\mu_1<1/2$. The function
\[
 F_-(t)=\frac{t^2}{1-2t},
 \qquad 0\le t<\frac12,
\]
is superadditive as long as the sum of the arguments is below $1/2$. Indeed,
\[
 F_-(x+y)-F_-(x)-F_-(y)
 =\frac{2xy(1-x-y)}
 {(1-2x)(1-2y)(1-2x-2y)}\ge0.
\]
Iterating this identity and using $\sum_{\alpha\ge2}\mu_\alpha=r$ gives
\[
 \Sigma_\mu\le\frac{r^2}{1-2r}.
\]
Consequently
\[
 1+\Sigma_\mu\le\frac{(1-r)^2}{1-2r}
 =\frac{\mu_1^2}{2\mu_1-1},
\]
and \eqref{eq:Q0-schur} is nonnegative. Thus $\mathscr Q_0\ge0$.

\smallskip
\noindent\emph{Group 2: the components with two zero indices.}
Fix $\alpha\in\{1,\dots,m\}$ and consider $T_{00\alpha}$ together with
\[
 x_\beta=T_{\beta\beta\alpha},
 \qquad 1\le\beta\le m.
\]
The trace constraint for $k=\alpha$ gives
\[
 T_{00\alpha}=-\sum_\beta x_\beta.
\]
After multiplication by the positive factor $\Lambda^2$, the contribution of this group is
\begin{equation}\label{eq:Qalpha}
 \mathscr Q_\alpha(x)
 =\sum_\beta d_{\alpha\beta}x_\beta^2
 -\frac{2-\mu_\alpha}{\mu_\alpha}
 \left(\sum_\beta x_\beta\right)^2,
\end{equation}
where
\begin{equation}\label{eq:dab}
 d_{\alpha\alpha}=\frac1{\mu_\alpha^2},
 \qquad
 d_{\alpha\beta}
 =\frac{\mu_\alpha+2\mu_\beta}
 {\mu_\alpha\mu_\beta^2}
 \quad(\beta\ne\alpha).
\end{equation}
By weighted Cauchy--Schwarz,
\[
 \left(\sum_\beta x_\beta\right)^2
 \le
 \left(\sum_\beta d_{\alpha\beta}x_\beta^2\right)
 \left(\sum_\beta d_{\alpha\beta}^{-1}\right).
\]
It remains to show
\begin{equation}\label{eq:d-inverse}
 \sum_\beta d_{\alpha\beta}^{-1}
 \le\frac{\mu_\alpha}{2-\mu_\alpha}.
\end{equation}
Let $\mu=\mu_\alpha$ and $r=1-\mu$. For fixed $\mu>0$, the function
\[
 F_\mu(t)=\frac{t^2}{\mu+2t}
\]
is superadditive on $[0,\infty)$, because
\[
 F_\mu(x+y)-F_\mu(x)-F_\mu(y)
 =\frac{2\mu xy(\mu+x+y)}
 {(\mu+2x)(\mu+2y)(\mu+2x+2y)}\ge0.
\]
Consequently,
\begin{align*}
 \sum_\beta d_{\alpha\beta}^{-1}
 &=\mu^2+
 \mu\sum_{\beta\ne\alpha}
 \frac{\mu_\beta^2}{\mu+2\mu_\beta}\\
 &\le \mu^2+\mu\frac{r^2}{\mu+2r}
 =\frac{\mu}{2-\mu},
\end{align*}
which proves \eqref{eq:d-inverse} and hence $\mathscr Q_\alpha\ge0$.

\smallskip
\noindent\emph{Group 3: one zero index and two distinct positive indices.}
If $\alpha\ne\beta$, the coefficient of $T_{0\alpha\beta}^2$ is
\begin{equation}\label{eq:mixed-coeff}
 2\left(
 \frac1{\lambda_\alpha\lambda_\beta}
 -\frac1{\Lambda\lambda_\alpha}
 -\frac1{\Lambda\lambda_\beta}
 \right)
 =\frac{2(\Lambda-\lambda_\alpha-\lambda_\beta)}
 {\Lambda\lambda_\alpha\lambda_\beta}
 \ge0.
\end{equation}
\smallskip
\noindent\emph{Group 4: three positive indices.}
Every remaining component has three indices in $\{1,\dots,m\}$, so its coefficient is positive. The four groups are disjoint and exhaust all components of the symmetric tensor. Their contributions are all nonnegative, and their sum is exactly the left-hand side of \eqref{eq:tensor-ineq}. This proves the lemma.
\end{proof}

\subsection{Superharmonicity of the logarithmic Hessian determinant}

\begin{theorem}[Gleason–Wolff logarithmic determinant inequality]\label{thm:det-min}
Let $\cO\subset\R^N$ be a domain and let $q\in C^4(\cO)$ be harmonic. Suppose that
\[
 n_-(D^2q)=1,
 \qquad n_0(D^2q)=0
 \quad\text{in }\cO.
\]
Then
\begin{equation}\label{eq:logdet-super}
 \Delta\log|\det D^2q|\le0
 \qquad\text{in }\cO.
\end{equation}
Consequently, if $\cO_0\Subset\cO$ is a bounded domain with piecewise $C^1$ boundary, then
\begin{equation}\label{eq:det-minimum}
 \inf_{\cO_0}|\det D^2q|
 \ge
 \inf_{\partial\cO_0}|\det D^2q|.
\end{equation}
\end{theorem}

\begin{proof}
Fix a point of $\cO$ and choose an orthonormal basis in which $H=D^2q$ is diagonal. Its eigenvalues can be written as
\[
 -\Lambda,\lambda_1,\dots,\lambda_{N-1},
 \qquad
 \lambda_\alpha>0.
\]
Since $q$ is harmonic,
\[
 \Lambda=\sum_{\alpha=1}^{N-1}\lambda_\alpha.
\]
Let $T_{ijk}=q_{ijk}$. Let $H_k=\partial_kH$ and $H_{kk}=\partial_{kk}H$. The standard derivative formula
\[
 \partial_k\log|\det H|=\tr(H^{-1}H_k)
\]
can be differentiated once more. Summing over $k$ gives
\begin{align*}
 \Delta\log|\det H|
 &=\sum_k\left
 \{-\tr(H^{-1}H_kH^{-1}H_k)
 +\tr(H^{-1}H_{kk})\right\}.
\end{align*}
The term containing fourth derivatives vanishes after the sum in $k$, because harmonicity commutes with differentiation:
\[
 \sum_k(H_{kk})_{ij}
 =\partial_{ij}\Delta q=0.
\]
Thus, in the eigenbasis of $H$,
\begin{equation}\label{eq:lap-logdet}
 \Delta\log|\det H|
 =-\sum_{i,j,k}\frac{T_{ijk}^2}{\lambda_i\lambda_j},
\end{equation}
where $\lambda_0=-\Lambda$. Moreover,
\[
 \sum_iT_{iik}=\partial_k\Delta q=0.
\]
The third derivatives form a completely symmetric tensor, and the last displayed trace identity is precisely the trace-free condition in \cref{lem:tensor}. That lemma says that the sum in \eqref{eq:lap-logdet} is nonnegative; the preceding minus sign therefore proves \eqref{eq:logdet-super}.

Since $\overline{\cO_0}\Subset\cO$, the function $\log|\det D^2q|$ is continuous on $\overline{\cO_0}$. The weak minimum principle for superharmonic functions gives
\[
 \inf_{\cO_0}\log|\det D^2q|
 \ge \inf_{\partial\cO_0}\log|\det D^2q|.
\]
Exponentiation proves \eqref{eq:det-minimum}.
\end{proof}

\subsection{The Gleason--Wolff zero-set theorem}

\begin{theorem}[Gleason--Wolff zero-set theorem]\label{thm:GW}
Let $\cO\subset\R^N$, $N\ge3$, be a  domain, and let $q$ be harmonic in $\cO$. Suppose that $D^2q$ has at most one negative eigenvalue at every point of $\cO$. If
\[
 \det D^2q(X_0)=0
\]
for some $X_0\in\cO$, then
\[
 \det D^2q\equiv0
 \qquad\text{in }\cO.
\]
\end{theorem}

This is the sign-symmetric formulation of \cite[Theorem~1]{GleasonWolff1991}, obtained by applying that theorem to $q$ or $-q$. It is qualitatively different from the logarithmic determinant inequality: it allows the Hessian to be singular a priori and says that one zero forces global vanishing on the connected component. We use only the following contrapositive form.

\begin{corollary}[Nonvanishing of the Hessian determinant]\label{cor:GW-propagation}
Let $q$ be harmonic in a  domain $\cO\subset\R^N$, $N\ge3$, and suppose that $D^2q$ has at most one negative eigenvalue everywhere. If $\det D^2q$ is nonzero at one point of $\cO$, then it is nonzero at every point of $\cO$.
\end{corollary}

\begin{proof}
A zero at any point would force the determinant to vanish identically by \cref{thm:GW}, contradicting the given nonsingular point.
\end{proof}

\begin{remark}
The logarithmic determinant inequality and the zero-set theorem have distinct roles. The first gives a quantitative lower bound while the perturbed Hessian is nonsingular; the second excludes zero eigenvalues after passage to a limit, when only the closed condition $n_+(D^2W)\le1$ remains.
\end{remark}

\section{Slit-edge asymptotics and Hessian inertia}\label{sec:edge}

\subsection{Admissible deformations of smooth uniformly convex domains}

We first assume that $D$ has $C^\infty$ boundary and strictly positive principal curvatures. The method of continuity uses the following class of domain paths. A family $\{D_\tau\}_{\tau\in[0,1]}$ is called \emph{admissible} if:

\begin{enumerate}[label=\textnormal{(F\arabic*)},leftmargin=2.7em]
\item every $D_\tau$ is bounded with $C^\infty$ boundary;
\item there exist balls $B_-\Subset B_+$ such that $B_-\subset D_\tau\subset B_+$ for all $\tau\in[0,1]$;
\item if $\nu_{\mathrm{in},\tau}$ denotes the inward unit normal, then there exist constants $0<\kappa_0\le\kappa_1<\infty$, independent of $\tau$, such that the shape operators
\[
 S_\tau=-D\nu_{\mathrm{in},\tau}\big|_{T\partial D_\tau}
\]
satisfy
\[
 \kappa_0I\le S_\tau\le\kappa_1I;
\]
\item for some fixed $\alpha\in(0,1)$, the local boundary charts and their inverses have a common radius, are uniformly bounded in $C^{6,\alpha}$, and depend continuously on $\tau$ in that topology;
\item $D_0=B_1$ and $D_1=D$.
\end{enumerate}

Such a family exists. By \eqref{eq:similarity-scaling}, translation and dilation of the endpoint domain do not affect the conclusion. Choose the origin in $D$ and let
\[
 h_D(\omega)=\sup_{x\in D}x\cdot\omega,
 \qquad \omega\in\Sn^{n-1},
\]
be its support function. Define
\[
 h_\tau=(1-\tau)h_{B_1}+\tau h_D.
\]
For a smooth strictly convex body, the tensor
\[
 \nabla^2_{\Sn^{n-1}}h+h g_{\Sn^{n-1}},
\]
where $g_{\Sn^{n-1}}$ is the standard metric, is the radius-of-curvature tensor. Along the path $h_\tau$, this tensor is a convex combination of two positive definite tensors. Its eigenvalues therefore admit positive upper and lower bounds on $[0,1]\times\Sn^{n-1}$. Smooth dependence of $h_\tau$ on $\tau$ gives the chart bounds in (F4), so the convex bodies with support functions $h_\tau$ form an admissible family.

\subsection{Parameter-uniform slit-edge expansion}

Fix $\tau$ and a point of $\partial D_\tau$. Let $z=(z_1,\dots,z_{n-1})$ be local coordinates on $\partial D_\tau$, let $\gamma_\tau$ be the corresponding parametrization, and let $\nu_{\mathrm{in},\tau}$ be the inward unit normal. Denote by $d_\tau$ the signed distance to $\partial D_\tau$, positive in $D_\tau$. The associated Fermi coordinates are
\begin{equation}\label{eq:fermi}
 x=\gamma_\tau(z)+s\nu_{\mathrm{in},\tau}(z),
 \qquad
 s>0\text{ in }D_\tau,
 \quad s<0\text{ in }D_\tau^c.
\end{equation}
For $|s|$ below the common tubular radius, $s=d_\tau(x)$. The Euclidean distance from $(x,y)$ to $\partial D_\tau\times\{0\}$ is
\[
 \rho=(s^2+y^2)^{1/2}.
\]
Set $D_{s,y}=(\partial_s,\partial_y)$ and
\[
 \theta=\arg(s+iy)\in(-\pi,\pi),
 \qquad
 h(s,y)=\rho^{1/2}\cos\frac\theta2
 =\Re(s+iy)^{1/2}.
\]
Thus the slit is represented in each two-dimensional normal plane by the negative $s$-axis $\{s\le0,y=0\}$. The function $h$ is harmonic away from this axis, vanishes on the two faces of the slit, is positive for $s>0$, and is homogeneous of degree $1/2$. It is the leading homogeneous Dirichlet singular function.

\begin{proposition}\label{prop:uniform-edge}
Let $\{D_\tau\}_{\tau\in[0,1]}$ be an admissible family. There exist constants $\rho_0>0$, $C<\infty$, and $a_0>0$, depending only on the uniform geometric data in (F1)--(F4), such that, in every Fermi chart,
\begin{equation}\label{eq:edge-expansion}
 W_{D_\tau}(z,s,y)=a_\tau(z)h(s,y)+R_\tau(z,s,y),
 \qquad 0<\rho<\rho_0,
\end{equation}
where
\begin{equation}\label{eq:edge-coefficient}
 \|a_\tau\|_{C^2(\partial D_\tau)}\le C,
 \qquad
 a_\tau(z)\ge a_0,
\end{equation}
and
\begin{align}
 |D_z^\beta R_\tau|&\le C\rho^{3/2},
 &&|\beta|\le2, \label{eq:edge-R-tt}\\
 |D_{s,y}R_\tau|+|D_zD_{s,y}R_\tau|&\le C\rho^{1/2},
 \label{eq:edge-R-tn}\\
 |D^2_{s,y}R_\tau|&\le C\rho^{-1/2}.
 \label{eq:edge-R-nn}
\end{align}
Derivatives at $y=0$, $s<0$, are understood as one-sided derivatives on each component of the slit domain. The estimates are uniform in the chart, in the edge point, and in $\tau$.
\end{proposition}

\begin{proof}
Fix the exponent $\alpha$ in (F4). We use the following local consequence of the Schauder theory in \cite{DeSilvaSavin2015}. Let $\Gamma\subset\R^n$ be a $C^{6,\alpha}$ hypersurface and set
\[
 \mathcal P=\{(x,y):y=0,\ d(x)\le0\},
 \qquad
 r=(d(x)^2+y^2)^{1/2},
 \qquad
 U_0=2^{-1/2}(d+r)^{1/2},
\]
where $d$ is the signed distance to $\Gamma$, with $d>0$ on the side complementary to $\mathcal P$. Suppose that $v\in C(B_1)$ is bounded and even in $y$, satisfies $\Delta v=0$ in $B_1\setminus\mathcal P$, and vanishes on $\mathcal P$. After the normalization of the edge graph used in \cite{DeSilvaSavin2015}, Theorem~3.3 there, with $k=4$, yields
\begin{equation}\label{eq:DS-expansion}
 v=U_0\left(\sum_{m=0}^{5}A_m(x)r^m+O(r^{5+\alpha})\right).
\end{equation}
Here $A_m\in C^{5-m,\alpha}$, and the term that is odd in $y$ in the general statement vanishes because $v$ is even. The coefficient norms and the pointwise remainder are bounded in terms of $n$, $\alpha$, $\|v\|_{L^\infty(B_1)}$, and the normalized $C^{6,\alpha}$ norm of $\Gamma$. We do not differentiate the pointwise $O$-term in \eqref{eq:DS-expansion}. Estimates for derivatives in the thin variables $x$ follow instead from the compatible tangent-polynomial estimates in \cite[Remark~3.2 and Proposition~6.2]{DeSilvaSavin2015}. More precisely, in the notation of that proposition, for each edge point $Z$ and each thin-space multi-index $\mu$ with $1\le|\mu|\le2$,
\begin{equation}\label{eq:DS-differentiated}
 D_x^\mu v(X)
 =\frac{U_0(X)}{r(X)^{2|\mu|-1}}
 \left(P_Z^\mu(x-Z,r(X))
 +O\bigl(|X-\overline Z|^{4+|\mu|+\alpha}\bigr)\right),
 \qquad \overline Z=(Z,0)\in\R^{n+1}.
\end{equation}
The polynomial $P_Z^\mu$ is obtained by formally differentiating the tangent polynomial for $v$ at $Z$; its monomial degrees range from $|\mu|-1$ to $4+|\mu|$. Its coefficient bounds and the constant in the remainder are uniform in $Z$ and have the same quantitative dependence as in \eqref{eq:DS-expansion}. The estimate holds first in the non-tangential region based at $\overline Z$ and, by the final localization step in Proposition~6.2, throughout a smaller slit neighborhood. Proposition~6.2 concerns only the thin variables $x$ and does not estimate derivatives in the extension variable $y$; those derivatives are obtained below from local Schauder estimates.

Write $x=\gamma(z)+s\nu(z)$ and define
\[
 a(z)=A_0(\gamma(z))
 =\lim_{s\downarrow0}\frac{v(z,s,0)}{\sqrt{s}}.
\]
The limit is intrinsic and hence defines the same function on overlapping charts. Since $A_0\in C^{5,\alpha}$, the function $a$ has a uniform $C^{5,\alpha}$ bound. Set $R=v-a(z)U_0$. Since
\[
 A_0(\gamma(z)+s\nu(z))-a(z)=sB(z,s)
\]
with two uniformly bounded tangential derivatives, let $\pi(x)$ denote nearest-point projection onto $\Gamma$ and define
\[
 E_4:=R-U_0\bigl\{A_0(x)-A_0(\pi(x))+rA_1(x)
       +r^2A_2(x)+r^3A_3(x)\bigr\}.
\]
The pointwise expansion \eqref{eq:DS-expansion} identifies this remainder as
\[
 E_4=U_0\{r^4A_4(x)+r^5A_5(x)\}+E,
 \qquad |E|\le CU_0r^{5+\alpha}.
\]
This identity is used only pointwise. The coefficient functions $A_0,\ldots,A_3$ have at least two classical derivatives; no classical derivative of $A_4$, $A_5$, or $E$ is taken. Instead, apply \eqref{eq:DS-differentiated} to $v$ at the nearest edge point and subtract the derivatives of the displayed $A_0,\ldots,A_3$ terms. The Fermi vector fields have uniformly bounded $C^{4,\alpha}$ coefficients, so conversion from Euclidean thin derivatives to $D_z$ and $\partial_s$ produces only lower-order terms of the same edge order. The compatibility of the tangent polynomials on overlapping non-tangential regions, including the cancellation of the tangential derivatives of $a$, then gives
\begin{equation}\label{eq:horizontal-R-estimates}
 |D_z^\beta\partial_s^jR|
 \le Cr^{3/2-j}
 \quad\text{if }|\beta|+j\le2.
\end{equation}
The $C^{5,\alpha}$ bound for $a$ is used below when taking one tangential derivative, although only the $C^2$ bound is recorded in the statement. The constants have the same quantitative dependence as those in \eqref{eq:DS-expansion}.

Extend $a$ constantly along the normal segments and denote the parallel hypersurfaces by $\Sigma_s=\{d=s\}$. In Fermi coordinates the Euclidean Laplacian has the form
\begin{equation}\label{eq:Fermi-Laplacian}
 L=\partial_s^2+\partial_y^2+b(z,s)\partial_s
   +\Delta_{\Sigma_s},
\end{equation}
where $\Delta_{\Sigma_s}$ is the Laplace--Beltrami operator of the induced metric on $\Sigma_s$, and $b$ and the coefficients of $\Delta_{\Sigma_s}$ have uniform $C^{3,\alpha}$ bounds. Since $Lv=0$,
\begin{equation}\label{eq:R-equation}
 LR=-L(aU_0).
\end{equation}
The principal normal part cancels because
\[
 (\partial_s^2+\partial_y^2)U_0=0.
\]
Since $U_0$ is independent of $z$ in Fermi coordinates and $a$ is independent of $s$,
\begin{equation}\label{eq:aU0-source}
 L(aU_0)=ab\,\partial_sU_0+U_0\Delta_{\Sigma_s}a.
\end{equation}
Fix a point $X$ whose distance from the edge is $r$. There is a constant $c>0$, depending only on the uniform geometric data, such that one of the following alternatives holds. Either a full coordinate ball of radius $cr$ centered at $X$ does not meet the exterior Dirichlet portion, or $X$ lies on the exterior side and is contained in a half-ball of radius $cr$, centered at the projection of $X$ onto $\{y=0\}$, whose closure is disjoint from the edge. Denote the applicable full ball or half-ball by $\mathscr B_{cr}(X)$. On this coordinate neighborhood, the right-hand side of \eqref{eq:R-equation} satisfies
\[
 \|L(aU_0)\|_{L^\infty(\mathscr B_{cr}(X))}
 +r^\alpha[L(aU_0)]_{C^\alpha(\mathscr B_{cr}(X))}
 \le Cr^{-1/2};
\]
this follows from
\[
 |D_{s,y}^jU_0|\le Cr^{1/2-j},
 \qquad j=0,1,2,
\]
together with the uniform $C^{5,\alpha}$ bound for $a$. In the full-ball case the interior Schauder estimate applies. In the half-ball case $R=0$ on the flat boundary portion, so the boundary Schauder estimate for the Dirichlet problem applies. When $s>0$, both $v$ and $R$ extend smoothly across $y=0$, and the point is covered by the full-ball case.

After dilation by $r$, \eqref{eq:horizontal-R-estimates}, \eqref{eq:R-equation}, and these Schauder estimates yield
\begin{equation}\label{eq:y-R-estimates}
 |D_{s,y}R|\le Cr^{1/2},
 \qquad
 |D^2_{s,y}R|\le Cr^{-1/2}.
\end{equation}
Let $\mathscr B'$ denote the concentric ball or half-ball of radius $cr/2$. On $\mathscr B'$, the same argument gives the full scale-invariant estimate
\[
 \|R\|_{L^\infty(\mathscr B')}+r\|DR\|_{L^\infty(\mathscr B')}
 +r^2\|D^2R\|_{L^\infty(\mathscr B')}
 +r^{2+\alpha}[D^2R]_{C^\alpha(\mathscr B')}
 \le Cr^{3/2}.
\]
Let $T$ be a Fermi tangential vector field. Differentiating \eqref{eq:R-equation} gives
\[
 L(TR)=-T L(aU_0)-[T,L]R.
\]
The commutator $[T,L]R$ is a linear combination of first and second derivatives of $R$ with coefficients having uniform $C^{1,\alpha}$ bounds. Estimate \eqref{eq:horizontal-R-estimates}, the full scaled $C^{2,\alpha}$ estimate above, and the coefficient bounds supplied by the choice $k=4$ show that the right-hand side has scale-invariant $C^\alpha$ norm bounded by $Cr^{-1/2}$. Moreover, $|TR|\le Cr^{3/2}$, and $TR=0$ on the exterior Dirichlet portion. Applying the corresponding interior or boundary Schauder estimate gives
\begin{equation}\label{eq:zy-R-estimate}
 |D_zD_{s,y}R|\le Cr^{1/2}.
\end{equation}
Estimates \eqref{eq:horizontal-R-estimates}--\eqref{eq:zy-R-estimate} imply \eqref{eq:edge-R-tt}--\eqref{eq:edge-R-nn}. The constants are independent of the ratio $|y|/r$, including in the limit toward either Dirichlet side of the slit.

We apply the local estimate to $v=W_{D_\tau}$. Every $D_\tau$ is a bounded $C^\infty$ domain and therefore a Lipschitz domain satisfying a uniform exterior ball condition. Applied for the fractional order $1/2$ and right-hand side $1$, \cite[Proposition~1.1]{RosOtonSerra2014} gives $u_{D_\tau}\in C^{1/2}(\R^n)$. Thus $W_{D_\tau}$ is continuous in each local slit neighborhood and vanishes continuously on the slit. By (F4), one common radius and one dilation reduce every edge chart to the normalized setting above. If $B_R$ contains all $D_\tau$, domain monotonicity and \eqref{eq:ball-torsion}, after scaling, imply
\[
 0\le u_{D_\tau}\le u_{B_R}.
\]
Positivity of the Poisson kernel and the bounded range of $y$ in the local coordinate neighborhoods yield a uniform local $L^\infty$ bound for $W_{D_\tau}$. The constants in the De Silva--Savin estimate and in the local Schauder estimates are therefore independent of $\tau$ and of the edge point. In the Fermi coordinates \eqref{eq:fermi}, one has $d=s$, $r=\rho$, and $U_0=h$, which proves \eqref{eq:edge-expansion} and the upper bounds in \eqref{eq:edge-coefficient}--\eqref{eq:edge-R-nn}. Only local slit neighborhoods are used; no boundedness assumption on the global slit is required.

To prove the lower bound in \eqref{eq:edge-coefficient}, let $r_*>0$ be a common interior rolling-ball radius, whose existence follows from (F2)--(F4). At an edge point, comparison with the interior tangent ball and domain monotonicity yield, for $0<s<r_*$,
\[
 u_{D_\tau}(z,s)\ge
 \kappa_n(2r_*s-s^2)^{1/2},
 \qquad
 \kappa_n=\frac{\Gamma(n/2)}{\sqrt\pi\,\Gamma((n+1)/2)}.
\]
Since $h(s,0)=\sqrt{s}$ and $R_\tau(z,s,0)=O(s^{3/2})$, division by $\sqrt{s}$ and passage to the limit give
\[
 a_\tau(z)\ge\kappa_n\sqrt{2r_*}=:a_0>0.
\]
This completes the proof.
\end{proof}

\subsection{Schur-complement analysis near the slit edge}

The normal Hessian of $h$ is
\begin{equation}\label{eq:h-hessian}
 D^2_{s,y}h
 =\frac1{4\rho^{3/2}}
 \begin{pmatrix}
 -\cos(3\theta/2)&-\sin(3\theta/2)\\
 -\sin(3\theta/2)&\cos(3\theta/2)
 \end{pmatrix}.
\end{equation}
Thus
\begin{equation}\label{eq:h-det}
 \det D^2_{s,y}h=-\frac1{16\rho^3},
 \qquad
 \In D^2_{s,y}h=(1,1).
\end{equation}
The $1/2$-homogeneity of $h$ gives
\begin{equation}\label{eq:euler-h}
 D^2h\binom{s}{y}=-\frac12\nabla h,
 \qquad
 \nabla h^{\mathsf T}(D^2h)^{-1}\nabla h=-h.
\end{equation}
Also
\begin{equation}\label{eq:hs}
 h_s=\frac{h}{2\rho}.
\end{equation}

Let
\[
 \Sigma_{\tau,s}:=\{x:d_\tau(x)=s\}
\]
and use an orthonormal Fermi frame adapted to $T_x\Sigma_{\tau,s}\oplus\Span\{\partial_s,\partial_y\}$. Relative to this splitting, write
\begin{equation}\label{eq:block-hessian}
 D^2V_{D_\tau,\eps}
 =\begin{pmatrix}
 T_\eps&B\\
 B^{\mathsf T}&N_\eps
 \end{pmatrix}.
\end{equation}
Let $S_{\tau,s}$ be the shape operator of $\Sigma_{\tau,s}$, with the convention $S_{\tau,0}>0$, and extend $a_\tau$ constantly along the normal segments. In the formulas below, $\nabla_{\Sigma_{\tau,s}}a_\tau$ and $\nabla^2_{\Sigma_{\tau,s}}a_\tau$ denote the tangential covariant gradient and Hessian on $\Sigma_{\tau,s}$. The leading terms in the three blocks are
\begin{align}
 T_0&=-a_\tau h_sS_{\tau,s}
 +h\nabla^2_{\Sigma_{\tau,s}}a_\tau,
 \label{eq:T0}\\
 B_0&=\nabla_{\Sigma_{\tau,s}}a_\tau\otimes\nabla_{s,y}h,
 \label{eq:B0}\\
 N_0&=a_\tau D^2_{s,y}h.
 \label{eq:N0}
\end{align}
The exact tangential--$s$ block also contains $hS_{\tau,s}\nabla_{\Sigma_{\tau,s}}a_\tau$. Since $h=O(\rho^{1/2})$, this term is included in the $O(\rho^{1/2})$ remainder in \eqref{eq:B-error}. The displayed blocks contain all leading-order terms.

By \eqref{eq:euler-h},
\begin{equation}\label{eq:exact-cancellation}
 B_0N_0^{-1}B_0^{\mathsf T}
 =-\frac{h}{a_\tau}\,
 \nabla_{\Sigma_{\tau,s}}a_\tau\otimes
 \nabla_{\Sigma_{\tau,s}}a_\tau.
\end{equation}
The Schur complement of $N_0$ in the model block matrix is
\begin{equation}\label{eq:model-schur}
 \mathscr S_0
 =-a_\tau h_sS_{\tau,s}
 +h\left(
 \nabla^2_{\Sigma_{\tau,s}}a_\tau+
 \frac1{a_\tau}\nabla_{\Sigma_{\tau,s}}a_\tau\otimes\nabla_{\Sigma_{\tau,s}}a_\tau
 \right).
\end{equation}
Since $h_s=h/(2\rho)$, $a_\tau\ge a_0$, $S_{\tau,s}\ge\kappa_0I/2$, and the term multiplied by $h$ in \eqref{eq:model-schur} is bounded in norm by $C$, there are uniform constants $c_0,\rho_1>0$ such that
\begin{equation}\label{eq:S0-negative}
 \mathscr S_0\le-c_0\frac{h}{\rho}I_{n-1}\le0
 \qquad(0<\rho<\rho_1).
\end{equation}
The first inequality is strict whenever $h>0$. As the slit edge is approached from the exterior Dirichlet face, $h\to0$; the term $-\eps I_{n-1}$ then gives the required uniform negative upper bound.

\begin{proposition}\label{prop:edge-inertia}
There exist constants $c_e,C_e>0$, depending only on the uniform geometric data in (F1)--(F4), such that, for every $\tau\in[0,1]$, $0<\eps\le1$, and $X\in\cG_{D_\tau}$ satisfying
\begin{equation}\label{eq:edge-radius}
 0<\rho=\dist(X,\partial D_\tau\times\{0\})<c_e\eps^2,
\end{equation}
one has
\begin{equation}\label{eq:edge-inertia}
 \In D^2V_{D_\tau,\eps}=(1,n)
\end{equation}
and
\begin{equation}\label{eq:edge-det-lower}
 |\det D^2V_{D_\tau,\eps}(X)|
 \ge C_e^{-1}\eps^{n-1}\rho^{-3}.
\end{equation}
The same estimates hold for the one-sided limits of the Hessian on the relative interiors of the two Dirichlet sides of the slit.
\end{proposition}

\begin{proof}
We compare the exact Hessian with the model blocks \eqref{eq:T0}--\eqref{eq:N0}. The edge expansion and the Euclidean Hessian formulas in Fermi coordinates give, in operator norm,
\begin{align}
 N_\eps&=N_0+\eps\diag(-1,n)+O(\rho^{-1/2}),
 \label{eq:N-error}\\
 B&=B_0+O(\rho^{1/2}),
 \label{eq:B-error}\\
 T_\eps&=T_0-\eps I_{n-1}+O(\rho^{1/2}).
 \label{eq:T-error}
\end{align}
The $O(\rho^{1/2})$ term in \eqref{eq:T-error} includes the second tangential derivatives of the remainder and the Fermi connection applied to its first normal derivative. The same order also absorbs the lower-order mixed term $hS_{\tau,s}\nabla_{\Sigma_{\tau,s}}a_\tau$ mentioned above.

From \eqref{eq:h-hessian}--\eqref{eq:N-error}, for $\rho$ small,
\begin{equation}\label{eq:N-inertia}
 \In N_\eps=(1,1),
 \qquad
 \|N_\eps^{-1}\|\le C\rho^{3/2},
 \qquad
 |\det N_\eps|\ge c\rho^{-3}.
\end{equation}
The inverse of the normal block changes only at lower order. Indeed, the resolvent identity $A^{-1}-B^{-1}=A^{-1}(B-A)B^{-1}$ gives
\begin{equation}\label{eq:N-inverse-diff}
 \|N_\eps^{-1}-N_0^{-1}\|
 \le C(\rho^{5/2}+\eps\rho^3).
\end{equation}
Since $\|B_0\|=O(\rho^{-1/2})$, equations \eqref{eq:B-error} and \eqref{eq:N-inverse-diff} imply
\begin{equation}\label{eq:schur-error}
 T_\eps-BN_\eps^{-1}B^{\mathsf T}
 =\mathscr S_0-\eps I_{n-1}+E,
 \qquad
 \|E\|\le C(\rho^{1/2}+\eps\rho^2).
\end{equation}
Choose $c_e$ in \eqref{eq:edge-radius} sufficiently small. Then $\rho<c_e\eps^2$ implies
\[
 \|E\|\le\frac\eps2.
\]
Together with \eqref{eq:S0-negative}, this yields
\begin{equation}\label{eq:schur-strict}
 T_\eps-BN_\eps^{-1}B^{\mathsf T}
 \le-\frac\eps2I_{n-1}.
\end{equation}
A symmetric block matrix with invertible normal block is congruent to the diagonal matrix formed by that block and its Schur complement. Sylvester's inertia formula therefore gives
\[
 \In D^2V_{D_\tau,\eps}
 =\In N_\eps+
 \In(T_\eps-BN_\eps^{-1}B^{\mathsf T})
 =(1,1)+(0,n-1)=(1,n).
\]
The block determinant identity, \eqref{eq:N-inertia}, and \eqref{eq:schur-strict} give \eqref{eq:edge-det-lower}.
\end{proof}

\section{Boundary and far-field Hessian estimates}\label{sec:boundary-far-field}

To apply the method of continuity in the truncated slit domain $\cG_{D_\tau}\cap B_M^N$, we require uniform nondegeneracy near the slit edge, near the exterior Dirichlet boundary $D_\tau^c\times\{0\}$ away from the edge, and near the truncation sphere $\partial B_M^N$. The first region was treated in the preceding section; this section treats the remaining two.

\subsection{Exterior Dirichlet boundary away from the slit edge}

Fix $M>0$ and $h_0>0$. For $x\in D_\tau^c$ with $|x|\le M$ and $\dist(x,D_\tau)\ge h_0$, set
\begin{equation}\label{eq:g}
 g_\tau(x)=c_n\int_{D_\tau}
 \frac{u_{D_\tau}(z)}{|x-z|^{N}}\,\dd z.
\end{equation}
The separation condition from $D_\tau$ permits differentiation under the integral, with bounds uniform in $\tau$. Since $U_{D_\tau}=0$ on this portion of $\{y=0\}$, its odd reflection is harmonic in a full neighborhood. Taylor expansion in $y$ gives
\begin{equation}\label{eq:exterior-taylor}
 U_{D_\tau}(x,y)
 =yg_\tau(x)-\frac{y^3}{6}\Delta g_\tau(x)+\mathcal R_\tau(x,y).
\end{equation}
Here, uniformly on the indicated exterior sets,
\[
 |D_x^\gamma\partial_y^\ell\mathcal R_\tau(x,y)|
 \le C_{M,h_0}|y|^{5-\ell}
 \qquad(|\gamma|+\ell\le2).
\]
Since $W_{D_\tau}-U_{D_\tau}=y$ for $y>0$, the Hessian of $V_{D_\tau,\eps}$ satisfies
\begin{equation}\label{eq:flat-hessian}
 D^2V_{D_\tau,\eps}(x,y)
 =\begin{pmatrix}
 -\eps I_n+yD^2g_\tau+O(y^3)&
 \nabla g_\tau+O(y^2)\\
 \nabla g_\tau^{\mathsf T}+O(y^2)&
 n\eps-y\Delta g_\tau+O(y^3)
 \end{pmatrix}.
\end{equation}

\begin{proposition}\label{prop:dirichlet-boundary}
Fix $M,h_0>0$ and $0<\eps\le1$. There exists $\eta=\eta(M,h_0,\eps)>0$, independent of $\tau$, such that, whenever
\[
 |x|\le M,
 \qquad
 x\in D_\tau^c,
 \qquad
 \dist(x,D_\tau)\ge h_0,
 \qquad
 0<|y|<\eta,
\]
one has
\[
 \In D^2V_{D_\tau,\eps}(x,y)=(1,n)
\]
and
\begin{equation}\label{eq:flat-det-lower}
 |\det D^2V_{D_\tau,\eps}(x,y)|
 \ge \frac n2\eps^{n+1}.
\end{equation}
\end{proposition}

\begin{proof}
The one-sided limit of the Hessian on this Dirichlet boundary is
\[
 \begin{pmatrix}
 -\eps I_n&p\\
 p^{\mathsf T}&n\eps
 \end{pmatrix},
 \qquad p=\nabla g_\tau(x).
\]
If $p=0$, the matrix equals $\diag(-\eps I_n,n\eps)$ and has inertia $(1,n)$. If $p\ne0$, the subspace $p^\perp\subset\R^n$ contributes $n-1$ eigenvalues equal to $-\eps$, while the restriction to $\Span\{p,e_y\}$ has determinant
\[
 -n\eps^2-|p|^2<0.
\]
In both cases the boundary Hessian has inertia $(1,n)$ and
\[
 |\det|=\eps^{n-1}(n\eps^2+|p|^2)
 \ge n\eps^{n+1}.
\]
The derivatives of $g_\tau$ are uniformly bounded on the indicated set, so the operator norm of the limiting matrix is bounded by $L=L(M,h_0,\eps)$. The determinant bound implies that every eigenvalue has absolute value at least $n\eps^{n+1}/L^n$. The Taylor estimate \eqref{eq:flat-hessian} therefore preserves the inertia and at least one half of the determinant bound when $0<|y|<\eta(M,h_0,\eps)$. For $y<0$, reflection conjugates the Hessian by $\diag(I_n,-1)$; congruence preserves both inertia and the absolute value of the determinant.
\end{proof}

\subsection{Far-field Hessian estimates}

All domains in an admissible family lie in a fixed ball and contain a fixed ball. For $X=(x,y)$, write $R=|X|$ and put
\begin{equation}\label{eq:mu}
 \mu_\tau=c_n\int_{D_\tau}u_{D_\tau}(z)\,\dd z.
\end{equation}
Comparison with fixed balls gives
\begin{equation}\label{eq:mu-bounds}
 0<\mu_-\le\mu_\tau\le\mu_+<\infty.
\end{equation}
Taylor expansion of the Poisson kernel in the bounded variable $z$, followed by differentiation under the integral, gives the uniform far-field expansion
\begin{equation}\label{eq:far-field-expansion}
 U_{D_\tau}(X)=\mu_\tau yR^{-N}+E_\tau(X),
 \qquad
 |D^kE_\tau(X)|\le C_kR^{-N-k}
 \quad(0\le k\le4).
\end{equation}
Let
\[
 \Phi_\tau(X)=\mu_\tau yR^{-N},
 \qquad
 v=\frac XR,
 \qquad
 e=e_y,
 \qquad
 \sigma=e\cdot v=\frac yR,
 \qquad
 K=N\mu_\tau R^{-N-1}.
\]
Differentiation gives
\begin{equation}\label{eq:far-hessian-model}
 D^2\Phi_\tau
 =K\{-e\otimes v-v\otimes e-\sigma I_N
 +(N+2)\sigma v\otimes v\}.
\end{equation}
If $x\ne0$, set $\omega=x/|x|$. On $\Span\{e_y,\omega\}^\perp$, the eigenvalue is $-K\sigma$ with multiplicity $n-1$, while
\begin{equation}\label{eq:far-plane-det}
 \det\left(D^2\Phi_\tau\big|_{\Span\{e_y,\omega\}}\right)
 =-K^2\{1+(n-1)\sigma^2\}<0.
\end{equation}
At $x=0$, formula \eqref{eq:far-hessian-model} reduces to
\[
 D^2\Phi_\tau(0,y)
 =N\mu_\tau y^{-N-1}\diag(-I_n,n).
\]
Consequently, $\In D^2\Phi_\tau=(1,n)$ for $y>0$.

Near the hyperplane $\{y=0\}$, the factor $\sigma$ in the angular eigenvalues must be retained. For $r=|x|$ large, \eqref{eq:g} satisfies
\begin{align}
 g_\tau(x)&=\mu_\tau r^{-N}+O(r^{-N-1}),
 \label{eq:g-far0}\\
 \nabla g_\tau(x)&=-N\mu_\tau r^{-N-1}\frac xr+O(r^{-N-2}),
 \label{eq:g-far1}\\
 D^2g_\tau(x)&=N\mu_\tau r^{-N-2}
 \left((N+2)\frac xr\otimes\frac xr-I_n\right)
 +O(r^{-N-3}).
 \label{eq:g-far2}
\end{align}

\begin{proposition}\label{prop:large-sphere}
There exist $M_0>1$, $c_\infty>0$, and $\delta_\infty>0$, depending only on the admissible family, such that for every $M\ge M_0$ and
\begin{equation}\label{eq:eps-far}
 0<\eps\le c_\infty\mu_-M^{-(n+2)},
\end{equation}
we have
\begin{equation}\label{eq:far-inertia}
 \In D^2V_{D_\tau,\eps}=(1,n)
\end{equation}
throughout
\[
 \{X:M-\delta_\infty<|X|<M+\delta_\infty\}\cap\cG_{D_\tau}.
\]
For every fixed pair $(M,\eps)$ satisfying \eqref{eq:eps-far}, there exists $\delta^{\infty}_{M,\eps}>0$ such that
\[
 |\det D^2V_{D_\tau,\eps}|\ge\delta^{\infty}_{M,\eps}
\]
in the same annular neighborhood, uniformly for $\tau\in[0,1]$.
\end{proposition}

\begin{proof}
By reflection symmetry it suffices to work in $y\ge0$. Choose $L>0$ sufficiently large, independently of $M$, and cover the annulus by $\{y\ge L\}$ and $\{0\le y\le2L\}$. In the first region the leading term in \eqref{eq:far-field-expansion} controls the eigenvalues tangent to horizontal spheres. In the second region we use Taylor expansion from the exterior Dirichlet portion of $\{y=0\}$.

\smallskip
\noindent\emph{Case 1: $y\ge L$.}
Assume $R\simeq M$. For $r=|x|>0$, let $\omega=x/r$. The orthogonal complement of $\Span\{e_y,\omega\}$ has dimension $n-1$, and the model Hessian in \eqref{eq:far-hessian-model} equals $-K\sigma I$ on this space. The determinant of its restriction to $\Span\{e_y,\omega\}$ is given by \eqref{eq:far-plane-det}. At $r=0$, the same estimates follow from the displayed diagonal formula above. Since the norm of the two-dimensional block is at most $CK$ and the absolute value of its determinant is at least $K^2$, both singular values are bounded below by $cK$.

The eigenvalues in the horizontal spherical directions satisfy
\[
 K\sigma\ge cK\frac{L}{M},
\]
whereas \eqref{eq:far-field-expansion} gives
\[
 \|D^2E_\tau\|\le C\frac{K}{M}.
\]
Choose $L$ so that $CK/M\le\frac12K\sigma$ in the horizontal spherical directions. Then choose $M_0$ large and $c_\infty$ small. Condition \eqref{eq:eps-far} implies $\eps\le cK$, so the perturbation of the two-dimensional block by $D^2E_\tau+\eps D^2Q_n$ has norm smaller than its spectral gap. In the horizontal spherical directions, the contribution of $\eps D^2Q_n$ equals $-\eps I$ and is nonpositive.

\smallskip
\noindent\emph{Case 2: $0\le y\le2L$.}
Here $r\simeq M$. Taylor expansion from the exterior Dirichlet boundary gives
\begin{align}
 D^2_{xx}U_{D_\tau}
 &=yD^2g_\tau+O\left(\frac{Gy^3}{r^3}\right),
 \label{eq:far-taylor-xx}\\
 \nabla_x(U_{D_\tau})_y
 &=\nabla g_\tau+O\left(\frac{Gy^2}{r^2}\right),
 \label{eq:far-taylor-xy}\\
 (U_{D_\tau})_{yy}
 &=-y\Delta g_\tau+O\left(\frac{Gy^3}{r^3}\right),
 \label{eq:far-taylor-yy}
\end{align}
where $G=|\nabla g_\tau|$ and all constants are uniform. By \eqref{eq:g-far1}, after increasing $M_0$ if necessary, $G>0$ and
\begin{equation}\label{eq:G-far}
 G=N\mu_\tau r^{-N-1}(1+O(r^{-1})).
\end{equation}
Set
\[
 \nu=\frac{\nabla g_\tau}{G},
 \qquad
 \mathcal H_\nu:=\nu^\perp\subset\R^n.
\]
Let $P_{\mathcal H_\nu}$ denote the orthogonal projection of $\R^n$ onto $\mathcal H_\nu$. Equations \eqref{eq:g-far1}--\eqref{eq:g-far2} give
\begin{equation}\label{eq:E-block-g}
 D^2g_\tau|_{\mathcal H_\nu}\le-\frac{G}{2r}I_{\mathcal H_\nu},
 \qquad
 \|P_{\mathcal H_\nu}D^2g_\tau\nu\|\le C\frac{G}{r^2}.
\end{equation}

Decompose the extended Hessian relative to
\[
 \mathcal H_\nu\oplus\Span\{\nu,e_y\}.
\]
From \eqref{eq:far-taylor-xx} and \eqref{eq:E-block-g}, its $\mathcal H_\nu$-block satisfies
\begin{equation}\label{eq:E-block}
 T_{\mathcal H_\nu}\le-c\left(\eps+\frac{Gy}{r}\right)I_{\mathcal H_\nu}
\end{equation}
for $M$ large. The $2\times2$ block $A$ on $\Span\{\nu,e_y\}$ has off-diagonal entry
\[
 G+O\left(\frac{Gy^2}{r^2}\right)
\]
and diagonal entries bounded by $C(\eps+Gy/r)$. Choose $c_0>0$ sufficiently small. Since $y\le2L$, taking $M$ large and imposing $\eps\le c_0G$ yields
\begin{equation}\label{eq:A-plane}
 \det A\le-\frac12G^2,
 \qquad
 \In A=(1,1),
 \qquad
 \|A^{-1}\|\le\frac CG.
\end{equation}
The mixed block from $\mathcal H_\nu$ into this plane is bounded by
\[
 C\frac{G(y+y^2)}{r^2}.
\]
The Schur correction to the $\mathcal H_\nu$-block has norm at most
\[
 C\frac{G(y+y^2)^2}{r^4}I_{\mathcal H_\nu},
\]
which is absorbed by \eqref{eq:E-block} uniformly for $0\le y\le2L$ once $M$ is large. Sylvester's inertia formula gives $(1,n)$ in this region.

Finally, \eqref{eq:G-far} and $r\simeq M$ imply $G\simeq\mu_\tau M^{-(n+2)}$. After decreasing $c_\infty$, condition \eqref{eq:eps-far} implies every smallness condition used above. The two regions overlap on $L\le y\le2L$, and the corresponding spectral gaps are bounded below uniformly in $\tau$. After decreasing $\delta_\infty$, they remain valid on an annulus of fixed width. Odd reflection across the exterior Dirichlet boundary gives a continuous extension of the Hessian estimates to the closure. Compactness after pullback then yields the lower bound $\delta^{\infty}_{M,\eps}>0$.
\end{proof}

\begin{lemma}\label{lem:far-point}
Let $D\subset\R^n$ be any nonempty bounded domain and let $W_D$ be the reflected extension of its half-Laplacian torsion function. Put $N=n+1$ and
\[
 \mu_D=c_n\int_Du_D(z)\,\dd z>0.
\]
Then, after fixing an origin in $\R^n$, there exists $R_D<\infty$ such that
\begin{equation}\label{eq:far-point-inertia}
 \In D^2W_D(0,R)=(1,n)
 \qquad(R\ge R_D).
\end{equation}
In particular, $\det D^2W_D$ is nonzero at some point of the upper half-space.
\end{lemma}

\begin{proof}
The total mass $\mu_D$ is positive because $u_D\ge0$ and $u_D\not\equiv0$. Since $D$ is bounded, the far-field expansion \eqref{eq:far-field-expansion}, with constants depending on $D$, gives along $X_R=(0,R)$
\[
 U_D(X)=\mu_D y|X|^{-N}+E_D(X),
 \qquad
 |D^2E_D(X_R)|\le C_DR^{-N-2}.
\]
The affine term in $W_D=U_D+y$ has zero Hessian. A direct differentiation at $X_R$ yields
\begin{equation}\label{eq:far-point-matrix}
 D^2W_D(X_R)
 =N\mu_DR^{-N-1}
 \begin{pmatrix}
 -I_n&0\\
 0&n
 \end{pmatrix}
 +O(R^{-N-2}).
\end{equation}
The leading matrix has one positive and $n$ negative eigenvalues, all separated from zero by a multiple of $R^{-N-1}$. The error is smaller by a factor $R^{-1}$, so \eqref{eq:far-point-inertia} follows for sufficiently large $R$.
\end{proof}

\section{Explicit Poisson extension and Hessian inertia for the unit ball}\label{sec:ball}

Let $B=B_1(0)\subset\R^n$ and put
\begin{equation}\label{eq:kappa}
 \kappa_n=\frac{\Gamma(n/2)}{\sqrt\pi\,\Gamma((n+1)/2)}.
\end{equation}
The half-Laplacian torsion function of the ball is
\begin{equation}\label{eq:ball-torsion}
 u_B(r)=\kappa_n\sqrt{1-r^2},
 \qquad 0\le r<1.
\end{equation}
This normalization follows from \cite[Theorem~1, with $\alpha=1$ and $p=1/2$]{Dyda2012}, after accounting for the opposite sign convention for the fractional Laplacian, or directly from the extension computation below.

For $y>0$, introduce oblate spheroidal coordinates
\begin{equation}\label{eq:oblate}
 r^2=(1+\xi^2)(1-\eta^2),
 \qquad
 y=\xi\eta,
 \qquad
 \xi\ge0,
 \quad 0\le\eta\le1.
\end{equation}
These coordinates are adapted to the mixed boundary decomposition: $B\times\{0\}$ corresponds to $\xi=0$, whereas $(\R^n\setminus\overline B)\times\{0\}$ corresponds to $\eta=0$. Define
\begin{equation}\label{eq:fn}
 f_n(\xi)=\kappa_n\xi\int_\xi^\infty
 \frac{\dd s}{s^2(1+s^2)^{n/2}}.
\end{equation}

\begin{lemma}\label{lem:ball-extension}
The Poisson extension of \eqref{eq:ball-torsion} is
\begin{equation}\label{eq:ball-extension}
 U_B(r,y)=\eta f_n(\xi).
\end{equation}
Moreover,
\begin{equation}\label{eq:Kxi}
 K(\xi):=f_n(\xi)-\xi f_n'(\xi)
 =\kappa_n(1+\xi^2)^{-n/2}>0.
\end{equation}
\end{lemma}

\begin{proof}
For an axisymmetric function of the separated form $U=\eta f(\xi)$, substitution of \eqref{eq:oblate} into the Euclidean Laplacian gives
\[
 \Delta U
 =\frac{\eta}{\xi^2+\eta^2}
 \left((1+\xi^2)f''+n\xi f'-nf\right).
\]
Differentiating \eqref{eq:fn} verifies the resulting ODE. To check the boundary values and normal derivative, integrate by parts:
\[
 \int_\xi^\infty\frac{\dd s}{s^2(1+s^2)^{n/2}}
 =\frac{(1+\xi^2)^{-n/2}}{\xi}
 -n\int_\xi^\infty(1+s^2)^{-(n+2)/2}\,\dd s.
\]
The beta integral at $\xi=0$ yields $f_n(0)=\kappa_n$ and $f_n'(0)=-1$. Hence \eqref{eq:ball-extension} is harmonic in the upper half-space, has boundary trace $u_B$ on $\R^n$, and tends to zero at infinity. Uniqueness of the Poisson extension identifies it with $U_B$. Its normal derivative is $-1$ on the interior disk, in agreement with \eqref{eq:torsion}. Differentiating \eqref{eq:fn} also gives \eqref{eq:Kxi}.
\end{proof}

Put $\delta_{\mathrm{obl}}=\xi^2+\eta^2$. The inverse coordinate derivatives are
\begin{equation}\label{eq:oblate-derivatives}
 \partial_r
 =\frac{\xi r}{\delta_{\mathrm{obl}}}\partial_\xi
 -\frac{\eta r}{\delta_{\mathrm{obl}}}\partial_\eta,
 \qquad
 \partial_y
 =\frac{\eta(1+\xi^2)}{\delta_{\mathrm{obl}}}\partial_\xi
 +\frac{\xi(1-\eta^2)}{\delta_{\mathrm{obl}}}\partial_\eta.
\end{equation}
Because $U_B$ is radial in the $x$ variables, every direction tangent to a horizontal sphere is an eigenvector of its Hessian. The corresponding $n-1$ angular eigenvalues are
\begin{equation}\label{eq:ball-angular}
 \lambda_{\mathrm{ang}}
 =\frac{(U_B)_r}{r}
 =-\frac{\eta K(\xi)}{\delta_{\mathrm{obl}}}<0.
\end{equation}
For the $(r,y)$-Hessian block
\[
 A_0=\begin{pmatrix}
 (U_B)_{rr}&(U_B)_{ry}\\
 (U_B)_{ry}&(U_B)_{yy}
 \end{pmatrix},
\]
direct differentiation, followed by elimination of $f_n''$ through the ODE, gives
\begin{equation}\label{eq:ball-plane-det}
 \det A_0
 =-\frac{(1+\xi^2)K(\xi)^2
 \{1+(n-1)\eta^2\}}{\delta_{\mathrm{obl}}^3}<0.
\end{equation}
The algebra leading to \eqref{eq:ball-plane-det} is recorded in \cref{app:ball-algebra}.

\begin{proposition}\label{prop:ball-inertia}
For every $n\ge2$,
\begin{equation}\label{eq:ball-inertia}
 \In D^2W_B=(1,n)
 \qquad\text{throughout }\cG_B.
\end{equation}
For every fixed $M>2$, there exists $\eps_B(M)>0$ such that
\begin{equation}\label{eq:perturbed-ball}
 \In D^2(W_B+\eps Q_n)=(1,n)
 \qquad\text{in }\cG_B\cap B_M^N
\end{equation}
whenever $0\le\eps\le\eps_B(M)$.
\end{proposition}

\begin{proof}
For $y>0$, the Hessian splits orthogonally into the $n-1$ angular directions and the plane spanned by the horizontal radial direction and $e_y$. Equation \eqref{eq:ball-angular} makes the angular eigenvalues strictly negative, while \eqref{eq:ball-plane-det} gives a negative determinant on the two-dimensional plane and hence one positive and one negative eigenvalue there. At $y<0$, reflection conjugates the Hessian by $\diag(I_n,-1)$. On the interior trace set $B\times\{0\}$, $W_B$ is smooth and even in $y$; hence the mixed entries vanish, $D_x^2u_B<0$, and
\[
 (W_B)_{yy}=-\Delta_xu_B>0.
\]
Thus \eqref{eq:ball-inertia} holds throughout the slit domain, including $B\times\{0\}$.

We turn to the perturbation. The angular eigenvalues become $\lambda_{\mathrm{ang}}-\eps<0$, and the $(r,y)$-Hessian block is
\[
 A_\eps=\begin{pmatrix}
 (U_B)_{rr}-\eps&(U_B)_{ry}\\
 (U_B)_{ry}&(U_B)_{yy}+n\eps
 \end{pmatrix}.
\]
Consequently,
\begin{equation}\label{eq:ball-det-eps}
 \det A_\eps
 =\det A_0+\eps\{n(U_B)_{rr}-(U_B)_{yy}\}-n\eps^2.
\end{equation}
Define, for $y>0$,
\[
 \mathcal R_B(r,y)
 :=\frac{\{n(U_B)_{rr}-(U_B)_{yy}\}^+}{-\det A_0},
 \qquad
 \mathcal M_M:=\{(r,y):r\ge0,\ y\ge0,\ r^2+y^2\le M^2\}.
\]
We show that $\mathcal R_B$ has a bounded one-sided extension to the compact meridional set $\mathcal M_M$.

On $y=0$, $0\le r<1$, the limiting block is
\[
 \diag\bigl(u_B''(r),-\Delta_xu_B(r)\bigr),
\]
so its determinant is strictly negative. On $y=0$, $r>1$, odd reflection gives
\[
 U_B(r,y)=yG_B(r)+O(y^3),
 \qquad
 G_B(r)=\frac{f_n(\sqrt{r^2-1})}{\sqrt{r^2-1}}.
\]
Writing $\xi=(r^2-1)^{1/2}$ and using \eqref{eq:Kxi}, one obtains
\[
 \frac{\dd}{\dd\xi}\left(\frac{f_n(\xi)}{\xi}\right)
 =-\frac{K(\xi)}{\xi^2}<0;
\]
hence $G_B'(r)<0$. The limiting $(r,y)$-Hessian block is
\[
 \begin{pmatrix}0&G_B'(r)\\G_B'(r)&0\end{pmatrix},
\]
and its determinant equals $-(G_B')^2<0$, while both diagonal entries tend to zero. Radial smoothness yields a finite limit of $\mathcal R_B$ as $r\downarrow0$. Near the slit edge, put
\[
 \rho_B=((r-1)^2+y^2)^{1/2}.
\]
Equation \eqref{eq:ball-plane-det} gives $-\det A_0\simeq\rho_B^{-3}$, whereas differentiation of the leading square-root term gives
\[
 |(U_B)_{rr}|+|(U_B)_{yy}|=O(\rho_B^{-3/2});
\]
hence $\mathcal R_B\to0$ at the edge. On the remaining portions of $\mathcal M_M$, the quotient and its one-sided boundary limits are continuous. Therefore
\[
 C_M:=\sup_{\mathcal M_M}\mathcal R_B<\infty,
\]
where $\mathcal R_B$ denotes the extension just described.

Choose
\[
 \eps_B(M)\le\min\left\{1,\frac{1}{2C_M}\right\},
\]
with the second restriction omitted if $C_M=0$. Then \eqref{eq:ball-det-eps} yields
\[
 \det A_\eps
 \le-\bigl(1-\eps C_M\bigr)(-\det A_0)-n\eps^2<0.
\]
Thus $A_\eps$ has inertia $(1,1)$, and together with the negative angular eigenvalues this proves \eqref{eq:perturbed-ball}.
\end{proof}

\section[C2-stability under smooth domain deformations]{\texorpdfstring{$C^2$}{C2}-stability under smooth domain deformations}\label{sec:domain-stability}

To compare Hessians corresponding to different parameters, we identify the domains by diffeomorphisms and distinguish the pullback Hessian from the Euclidean Hessian of the pullback.

\begin{proposition}\label{prop:C2-continuity}
Let $\{D_\tau\}$ be an admissible family. There exists $R_*>0$ with the following property. For every fixed $\tau_*\in[0,1]$, there are diffeomorphisms $\Psi_\tau:\R^N\to\R^N$, defined for $\tau$ near $\tau_*$, such that $\Psi_{\tau_*}=\Id$, $\Psi_\tau(\cG_{D_{\tau_*}})=\cG_{D_\tau}$, $\Psi_\tau=\Id$ on $(\R^n\setminus B_{R_*}^n)\times\R$, and
\begin{equation}\label{eq:original-hessian-continuity}
 D^2W_{D_\tau}\circ\Psi_\tau
 \longrightarrow D^2W_{D_{\tau_*}}
 \quad\text{uniformly on }K
\end{equation}
for every compact set $K\Subset\cG_{D_{\tau_*}}$ satisfying
\[
 \dist(K,\partial D_{\tau_*}\times\{0\})>0.
\]
The corresponding one-sided convergence holds up to compact subsets of $(\R^n\setminus\overline{D_{\tau_*}})\times\{0\}$. The same assertions hold with $W_{D_\tau}$ replaced by $V_{D_\tau,\eps}$ for fixed $\eps\ge0$.
\end{proposition}

\begin{proof}
For $\tau$ sufficiently close to $\tau_*$, the hypersurface $\partial D_\tau$ is a normal graph over $\partial D_{\tau_*}$. Extending the normal graph map by a smooth cut-off supported in the common tubular neighborhood produces orientation-preserving diffeomorphisms
\[
 F_\tau:\R^n\to\R^n,
 \qquad
 F_\tau(D_{\tau_*})=D_\tau,
\]
for which there is a constant $R_*>0$, independent of $\tau_*$, such that $F_\tau=\Id$ on $\R^n\setminus B_{R_*}^n$. They converge to $F_{\tau_*}=\Id$ in $C^{4,\alpha}$ as $\tau\to\tau_*$. Set
\[
 \Psi_\tau(z,y)=(F_\tau(z),y),
 \qquad
 \widetilde U_\tau=U_{D_\tau}\circ\Psi_\tau,
 \qquad
 \widetilde W_\tau=W_{D_\tau}\circ\Psi_\tau.
\]
The map $\Psi_\tau$ sends the fixed slit onto the moving slit. Since $N\ge3$, define
\[
 \dot H^1(\R^N_+)
 :=\{v\in L^{2N/(N-2)}(\R^N_+):\nabla v\in L^2(\R^N_+)\},
\]
with norm $\|\nabla v\|_{L^2}$. Pullback identifies the moving variational spaces with
\[
 \mathscr H_{\tau_*}
 =\{v\in\dot H^1(\R^N_+):
 \operatorname{Tr}v\in H^{1/2}_0(D_{\tau_*})\}.
\]
Writing $J_\tau=\det DF_\tau$, the weak extension problem becomes
\begin{equation}\label{eq:pullback-weak}
 \int_{\R^N_+}A_\tau\nabla\widetilde U_\tau\cdot\nabla\zeta
 =\int_{D_{\tau_*}}J_\tau(z)\zeta(z,0)\,\dd z,
 \qquad \zeta\in\mathscr H_{\tau_*},
\end{equation}
where
\begin{equation}\label{eq:pullback-A}
 A_\tau
 =J_\tau
 \begin{pmatrix}
 DF_\tau^{-1}(DF_\tau)^{-\mathsf T}&0\\
 0&1
 \end{pmatrix}.
\end{equation}
The matrices $A_\tau$ are uniformly elliptic because $DF_\tau$ and its inverse are uniformly bounded. Testing \eqref{eq:pullback-weak} with $\widetilde U_\tau$ and using the trace Sobolev inequality gives a uniform energy bound.

To prove continuity, put $w=\widetilde U_\tau-\widetilde U_\sigma$, subtract the two weak equations, and use $w$ itself as the test function. Uniform ellipticity, the trace Sobolev inequality, and the uniform energy bound give
\begin{align*}
 c\|\nabla w\|_{L^2}^2
 &\le
 \left|\int_{\R^N_+}(A_\sigma-A_\tau)
 \nabla\widetilde U_\sigma\cdot\nabla w\right|
 +\left|\int_{D_{\tau_*}}(J_\tau-J_\sigma)w(z,0)\,\dd z\right|\\
 &\le C\bigl(\|A_\tau-A_\sigma\|_{L^\infty}
 +\|J_\tau-J_\sigma\|_{L^\infty}\bigr)
 \|\nabla w\|_{L^2}.
\end{align*}
Consequently,
\begin{equation}\label{eq:energy-shape-continuity}
 \|\nabla(\widetilde U_\tau-\widetilde U_\sigma)\|_{L^2(\R^N_+)}
 \longrightarrow0
 \qquad(\tau\to\sigma).
\end{equation}

Let $K\Subset\cG_{D_{\tau_*}}$ satisfy $\dist(K,\partial D_{\tau_*}\times\{0\})>0$. On subsets of $\{y\ne0\}$, interior estimates apply directly. Near $D_{\tau_*}\times\{0\}$, use the even reflection of $\widetilde W_\tau$; near $(\R^n\setminus\overline{D_{\tau_*}})\times\{0\}$, use the odd reflection of $\widetilde U_\tau$. The matrix in \eqref{eq:pullback-A} is block diagonal and independent of $y$, so the reflected functions satisfy uniformly elliptic divergence-form equations with uniform $C^{2,\alpha}$ coefficient bounds. The local $L^2$ convergence following from \eqref{eq:energy-shape-continuity}, combined with interior or flat Dirichlet Schauder estimates, gives
\[
 \widetilde W_\tau\longrightarrow\widetilde W_{\tau_*}
 \quad\text{in }C^{2,\beta}(K)
 \qquad(0<\beta<\alpha),
\]
with the analogous one-sided convergence up to the exterior Dirichlet boundary.

Finally,
\begin{equation}\label{eq:hessian-transform}
 D^2(W_{D_\tau}\circ\Psi_\tau)
 =D\Psi_\tau^{\mathsf T}(D^2W_{D_\tau}\circ\Psi_\tau)D\Psi_\tau
 +\sum_{a=1}^{N}
 (\partial_aW_{D_\tau}\circ\Psi_\tau)D^2\Psi_\tau^a.
\end{equation}
The second term records the failure of the Euclidean Hessian to transform tensorially under nonlinear changes of variables. The convergence of $\widetilde W_\tau$ in $C^{2,\beta}$, the convergence of $\Psi_\tau$ in $C^2$, and uniform invertibility of $D\Psi_\tau$ allow \eqref{eq:hessian-transform} to be solved for $D^2W_{D_\tau}\circ\Psi_\tau$. This proves \eqref{eq:original-hessian-continuity} and the one-sided boundary statement.
\end{proof}

\section{Continuation argument for smooth uniformly convex domains}\label{sec:global}

Increase $M_0$, if necessary, and fix $M\ge M_0$ such that $M>R_*+2$ and
\[
 \overline{D_\tau}\times\{0\}\subset B_{M-2}^N
 \qquad(0\le\tau\le1).
\]
Choose
\begin{equation}\label{eq:epsilon-choice}
 0<\eps\le
 \min\{1,\eps_B(M),c_\infty\mu_-M^{-(n+2)}\}.
\end{equation}
For this fixed pair $(M,\eps)$, define
\begin{equation}\label{eq:Tset}
 \cT_{M,\eps}
 =\left\{\tau\in[0,1]:
 \In D^2V_{D_\tau,\eps}=(1,n)
 \text{ at every point of }
 \cG_{D_\tau}\cap B_M^N
 \right\}.
\end{equation}
By \cref{prop:ball-inertia}, $0\in\cT_{M,\eps}$.

\subsection{Boundary nondegeneracy and the interior subdomain}

Write
\[
 \Omega_\tau^M:=\cG_{D_\tau}\cap B_M^N.
\]
For each of the slit edge, the exterior Dirichlet boundary away from the edge, and the truncation sphere, choose two ambient neighborhoods
\[
 \mathcal U_{\tau,\bullet}^-\Subset\mathcal U_{\tau,\bullet}^+,
 \qquad \bullet\in\{\mathrm e,\mathrm D,\infty\},
\]
such that $\overline{\mathcal U_{\tau,\bullet}^+}\cap\Omega_\tau^M$ is contained in the corresponding region on which the conclusion of \cref{prop:edge-inertia,prop:dirichlet-boundary,prop:large-sphere} holds, with one-sided interpretation at the exterior Dirichlet boundary. Choose them so that their smaller union covers the relative boundary of $\Omega_\tau^M$. They may be chosen to vary continuously after pullback by $\Psi_\tau$. Put
\[
 \mathcal U_\tau^\pm
 :=\bigcup_{\bullet\in\{\mathrm e,\mathrm D,\infty\}}
 \mathcal U_{\tau,\bullet}^\pm.
\]
Then
\begin{equation}\label{eq:boundary-neighborhood-cover}
 \partial\Omega_\tau^M\subset\mathcal U_\tau^-
 \Subset\mathcal U_\tau^+,
 \qquad
 \Omega_\tau^M\setminus\mathcal U_\tau^+
 \Subset\Omega_\tau^M\setminus\overline{\mathcal U_\tau^-}.
\end{equation}
The three propositions and compactness in the parameter $\tau$ give a constant
\begin{equation}\label{eq:boundary-det-delta}
 \delta_{M,\eps}>0
\end{equation}
such that, for every $\tau\in[0,1]$,
\begin{equation}\label{eq:boundary-det-lower}
 \In D^2V_{D_\tau,\eps}=(1,n),
 \qquad
 |\det D^2V_{D_\tau,\eps}|
 \ge\delta_{M,\eps}
\end{equation}
on $\overline{\mathcal U_\tau^+}\cap\Omega_\tau^M$, with one-sided interpretation on the exterior Dirichlet boundary.

Choose a bounded open set $\cK_\tau$, possibly with finitely many connected components and with piecewise $C^1$ boundary, such that
\begin{equation}\label{eq:core-sandwich}
 \Omega_\tau^M\setminus\mathcal U_\tau^+
 \subset\cK_\tau
 \subset\overline{\cK_\tau}
 \subset\Omega_\tau^M\setminus\overline{\mathcal U_\tau^-}.
\end{equation}
Smoothing finitely many intersections of level sets gives such an open set. The construction can be made with defining functions that depend continuously on $\tau$ after pullback and whose zero level sets are regular. Consequently, for every $\tau_*$, the sets $\Psi_\tau^{-1}(\overline{\cK_\tau})$ converge to $\overline{\cK_{\tau_*}}$ in the Hausdorff topology as $\tau\to\tau_*$. In particular, they are contained in a fixed compact subset of $\Omega_{\tau_*}^M$ when $\tau$ is close to $\tau_*$, and every limit point of a sequence in $\Psi_\tau^{-1}(\overline{\cK_\tau})$ belongs to $\overline{\cK_{\tau_*}}$. Moreover,
\[
 \partial\cK_\tau\subset\mathcal U_\tau^+\setminus
 \overline{\mathcal U_\tau^-},
\]
and \eqref{eq:boundary-det-lower} holds on $\partial\cK_\tau$.

\begin{lemma}\label{lem:interior-det-lower}
If $\tau\in\cT_{M,\eps}$, then
\begin{equation}\label{eq:interior-det-lower}
 |\det D^2V_{D_\tau,\eps}|
 \ge\delta_{M,\eps}
 \qquad\text{on }\overline{\cK_\tau}.
\end{equation}
\end{lemma}

\begin{proof}
Set $q_\tau=-V_{D_\tau,\eps}$. If $\tau\in\cT_{M,\eps}$, then $n_-(D^2q_\tau)=1$ and $n_0(D^2q_\tau)=0$ in $\Omega_\tau^M$. By \cref{thm:det-min},
\[
 \log|\det D^2q_\tau|
\]
is superharmonic on every connected component of $\cK_\tau$. The boundary lower bound \eqref{eq:boundary-det-lower} and the minimum principle give \eqref{eq:interior-det-lower}.
\end{proof}

\subsection{Open--closed argument}

\begin{proposition}\label{prop:T-open-closed}
The set $\cT_{M,\eps}$ is both open and closed in $[0,1]$.
\end{proposition}

\begin{proof}
\emph{Openness.}
Let $\tau_0\in\cT_{M,\eps}$. Suppose that $\tau_j\to\tau_0$ and that there exist $X_j\in\Omega_{\tau_j}^M$ at which the inertia is not $(1,n)$. By \eqref{eq:boundary-det-lower}, $X_j\notin\overline{\mathcal U_{\tau_j}^+}$; hence \eqref{eq:core-sandwich} gives $X_j\in\cK_{\tau_j}$. Use the diffeomorphisms based at $\tau_0$ and set
\[
 \widehat X_j=\Psi_{\tau_j}^{-1}(X_j).
\]
After passing to a subsequence, $\widehat X_j\to X_0\in\overline{\cK_{\tau_0}}$. By \cref{lem:interior-det-lower},
\[
 |\det D^2V_{D_{\tau_0},\eps}(X_0)|\ge\delta_{M,\eps},
\]
and this Hessian has inertia $(1,n)$. The convergence in \cref{prop:C2-continuity} then gives the same inertia at $X_j$ for all sufficiently large $j$, a contradiction. Thus $\cT_{M,\eps}$ is open.

\smallskip
\noindent\emph{Closedness.}
Let $\tau_j\in\cT_{M,\eps}$ and $\tau_j\to\tau_*$. Fix $X\in\Omega_{\tau_*}^M$ and set $X_j=\Psi_{\tau_j}(X)$. If $X\in\overline{\mathcal U_{\tau_*}^+}$, then \eqref{eq:boundary-det-lower}, together with the continuous pullback construction of these neighborhoods, gives the desired inertia at $X$. If $X\notin\overline{\mathcal U_{\tau_*}^+}$, continuity gives $X_j\notin\mathcal U_{\tau_j}^+$ for all sufficiently large $j$, and \eqref{eq:core-sandwich} implies $X_j\in\cK_{\tau_j}$. Hence
\[
 |\det D^2V_{D_{\tau_j},\eps}(X_j)|
 \ge\delta_{M,\eps}.
\]
The $C^2$ stability in \cref{prop:C2-continuity} gives
\[
 |\det D^2V_{D_{\tau_*},\eps}(X)|
 \ge\delta_{M,\eps}.
\]
The Hessian at $X$ is therefore the limit of matrices with inertia $(1,n)$ and determinant bounded away from zero. Continuity of the ordered eigenvalues gives $\In D^2V_{D_{\tau_*},\eps}(X)=(1,n)$. Since $X\in\Omega_{\tau_*}^M$ was arbitrary, $\tau_*\in\cT_{M,\eps}$, and $\cT_{M,\eps}$ is closed.
\end{proof}

Connectedness of $[0,1]$ and $0\in\cT_{M,\eps}$ give
\begin{equation}\label{eq:T-all}
 \cT_{M,\eps}=[0,1].
\end{equation}
In particular, for the endpoint domain $D=D_1$,
\begin{equation}\label{eq:perturbed-target-inertia}
 \In D^2(W_D+\eps Q_n)=(1,n)
 \qquad\text{in }\cG_D\cap B_M^N.
\end{equation}

\subsection[Limit as epsilon tends to zero]{Limit as $\eps\downarrow0$}

For fixed $M\ge M_0$, choose any sequence $\eps_k\downarrow0$ satisfying \eqref{eq:epsilon-choice}. The Hessians converge pointwise to $D^2W_D$. The set of symmetric matrices with at most one positive eigenvalue is closed, because ordered eigenvalues depend continuously on the matrix. Hence
\begin{equation}\label{eq:closed-one-positive-M}
 n_+(D^2W_D)\le1
 \qquad\text{in }\cG_D\cap B_M^N.
\end{equation}
The argument may be repeated for every $M\ge M_0$. Since every point of $\cG_D$ belongs to $B_M^N$ for all sufficiently large $M$,
\begin{equation}\label{eq:closed-one-positive}
 n_+(D^2W_D)\le1
 \qquad\text{throughout }\cG_D.
\end{equation}
The inequality $n_+(D^2W_D)\le1$ is stable under matrix limits but does not exclude zero eigenvalues. Nondegeneracy is proved in the next section.

\section{Nondegeneracy and strict concavity in the smooth uniformly convex case}\label{sec:smooth-strict}

\subsection{Nondegeneracy from the Gleason--Wolff zero-set theorem}

Choose $x_*\in D$. Every point in the upper, respectively lower, half-space can be joined within that half-space to $(x_*,1)$, respectively $(x_*,-1)$, and
\[
 \{x_*\}\times[-1,1]\subset\cG_D.
\]
Thus $\cG_D$ is path connected. Set
\[
 q=-W_D.
\]
By \eqref{eq:closed-one-positive}, $D^2q$ has at most one negative eigenvalue throughout $\cG_D$. On the other hand, \cref{lem:far-point} provides a point at which $D^2W_D$, and hence $D^2q$, is nonsingular. Therefore \cref{cor:GW-propagation} gives
\begin{equation}\label{eq:strict-extension-inertia-smooth}
 \det D^2W_D\ne0
 \qquad\text{throughout }\cG_D.
\end{equation}
Since $W_D$ is harmonic, $\tr D^2W_D=0$. The determinant is now nonzero, so no eigenvalue vanishes. A trace-free matrix cannot have all eigenvalues of the same sign. Combined with the bound of at most one positive eigenvalue, this leaves exactly one positive and $n$ negative eigenvalues.
Consequently,
\begin{equation}\label{eq:full-inertia-smooth}
 \In D^2W_D=(1,n)
 \qquad\text{throughout }\cG_D.
\end{equation}

\subsection[Positivity of the normal second derivative]{Positivity of $W_{yy}$ on $D\times\{0\}$}

For $y>0$, define
\begin{equation}\label{eq:Z}
 Z=(W_D)_y=1+(U_D)_y.
\end{equation}
The derivative of a harmonic function is harmonic, so $Z$ is harmonic in the upper half-space. On the interior trace set,
\begin{equation}\label{eq:Z-interior}
 Z=0\qquad\text{on }D\times\{0\}.
\end{equation}
On the relative interior of the exterior Dirichlet boundary, $x\in\R^n\setminus\overline D$,
\begin{equation}\label{eq:Z-exterior}
 Z(x,0+)
 =1+c_n\int_D\frac{u_D(z)}{|x-z|^N}\,\dd z>0.
\end{equation}
Moreover, $Z(X)\to1$ as $|X|\to\infty$. Near the slit edge, \eqref{eq:edge-expansion} gives
\[
 Z=a_1(z)\,\partial_y h+O(\rho^{1/2}),
 \qquad
 \partial_y h=\frac{\sin(\theta/2)}{2\sqrt\rho}\ge0
 \quad(0\le\theta\le\pi).
\]
For $0<\delta<\rho_0$, define the upper tubular neighborhood
\[
 \mathcal N_\delta^+
 :=\{(x,y):y>0,\ \dist((x,y),\partial D\times\{0\})<\delta\}.
\]
On $\partial\mathcal N_\delta^+\cap\{y>0\}$, the preceding expansion gives
\[
 Z\ge-C\delta^{1/2}.
\]
On the portion of $\{y=0\}$ outside $\mathcal N_\delta^+$, equations \eqref{eq:Z-interior}--\eqref{eq:Z-exterior} give $Z\ge0$. Since $Z(X)\to1$ as $|X|\to\infty$, one has $Z\ge1/2$ on the spherical boundary of a sufficiently large upper half-ball. Applying the minimum principle to
\[
 (B_R^N\cap\R^N_+)\setminus\overline{\mathcal N_\delta^+}
\]
gives
\[
 Z\ge-C\delta^{1/2}
\]
there. Letting $R\to\infty$ and then $\delta\downarrow0$ gives $Z\ge0$ in $\R^N_+$. Since $Z$ is harmonic and is not identically zero, the strong maximum principle gives
\begin{equation}\label{eq:Z-positive}
 Z>0\qquad\text{in }\R^N_+.
\end{equation}
Fix $x\in D$. In a small upper half-ball centered at $(x,0)$, the function $Z$ is positive inside and vanishes on the flat part of the boundary. The inward unit normal there is $+e_y$, so the Hopf boundary lemma gives
\begin{equation}\label{eq:Wyy-positive}
 (W_D)_{yy}(x,0)=Z_y(x,0)>0.
\end{equation}
Using \eqref{eq:bottom-block},
\[
 D^2W_D(x,0)
 =\begin{pmatrix}
 D_x^2u_D(x)&0\\
 0&(W_D)_{yy}(x,0)
 \end{pmatrix}.
\]
The extended Hessian has inertia $(1,n)$ by \eqref{eq:full-inertia-smooth}, and its $yy$ entry is positive by \eqref{eq:Wyy-positive}. Hence the $x$-Hessian block is negative definite:
\begin{equation}\label{eq:D2u-strict-smooth}
 D^2u_D(x)<0
 \qquad(x\in D).
\end{equation}
This proves \cref{thm:main} for smooth uniformly convex domains.

\section{Approximation by smooth uniformly convex domains}\label{sec:approximation}

Let $D\subset\R^n$ be an arbitrary bounded convex domain. Choose $x_0\in D$ and numbers $0<t_1<t_2<\cdots\uparrow1$. The compact convex bodies
\[
 K_j=x_0+t_j(\overline D-x_0)
\]
satisfy $K_j\subset\operatorname{int}K_{j+1}\Subset D$, and their interiors exhaust $D$. We record a support-function construction of the required smooth intermediate bodies; see also \cite[Section~3.3]{Schneider2014}. Set
\[
 \delta_j=\dist(K_j,\partial K_{j+1})>0
\]
and choose $\eta_j>0$ with $3\eta_j<\delta_j$. Let $h_j^{\mathrm{sm}}$ be a rotational convolution of the support function $h_{K_j}$, with the convolution scale chosen so that
\[
 \|h_j^{\mathrm{sm}}-h_{K_j}\|_{L^\infty(\Sn^{n-1})}<\eta_j.
\]
Define $L_j$ by the smooth support function
\[
 h_{L_j}=h_j^{\mathrm{sm}}+2\eta_j.
\]
The support-function inequalities imply
\[
 K_j\subset\operatorname{int}L_j
 \subset K_j+3\eta_jB_1
 \subset\operatorname{int}K_{j+1}.
\]
Moreover, the radius-of-curvature tensor satisfies
\[
 \nabla^2_{\Sn^{n-1}}h_{L_j}+h_{L_j}g_{\Sn^{n-1}}
 \ge2\eta_jg_{\Sn^{n-1}}>0.
\]
Consequently, $L_j$ is a $C^\infty$ convex body with strictly positive principal curvatures and
\[
 K_j\subset\operatorname{int}L_j,
 \qquad
 L_j\subset\operatorname{int}K_{j+1}.
\]
Since
\[
 L_j\subset\operatorname{int}K_{j+1}
 \subset\operatorname{int}L_{j+1},
\]
the domains $D_j:=\operatorname{int}L_j$ satisfy
\begin{equation}\label{eq:domain-approx}
 D_j\Subset D_{j+1}\Subset D,
 \qquad
 \bigcup_{j=1}^\infty D_j=D.
\end{equation}

Let $u_j=u_{D_j}$ and let $u=u_D$. By domain monotonicity,
\begin{equation}\label{eq:uj-monotone}
 0\le u_j\le u_{j+1}\le u.
\end{equation}
The closed subspaces $H^{1/2}_0(D_j)$ are nested. Their union is dense in $H^{1/2}_0(D)$: every compactly supported smooth function in $D$ is supported in $D_j$ for all sufficiently large $j$, and such functions are dense by definition. For $v\in H^{1/2}_0(D_j)$, the right-hand sides of the weak equations for $u$ and $u_j$ agree, because $v$ vanishes outside $D_j$. Hence
\[
 \cE(u-u_j,v)=0
 \qquad(v\in H^{1/2}_0(D_j)).
\]
Thus $u_j$ is the orthogonal projection of $u$ onto $H^{1/2}_0(D_j)$ with respect to the energy inner product. Orthogonal projections onto an increasing family of closed subspaces converge strongly to the projection onto the closure of their union. That closure is $H^{1/2}_0(D)$, so the fractional Poincar\'e inequality on the fixed bounded set $D$ gives
\begin{equation}\label{eq:energy-convergence}
 u_j\longrightarrow u
 \qquad\text{strongly in }H^{1/2}(\R^n).
\end{equation}
Because all functions vanish outside the same bounded set, \eqref{eq:energy-convergence} implies
\begin{equation}\label{eq:L1-convergence}
 \|u_j-u\|_{L^1(\R^n)}\longrightarrow0.
\end{equation}

Let $U_j,U$ be the corresponding Poisson extensions and put
\[
 W_j(x,y)=U_j(x,|y|)+|y|,
 \qquad
 W(x,y)=U(x,|y|)+|y|.
\]
If $K\Subset\R^N_+$ and $|\gamma|\le2$, then the derivatives $D_X^\gamma$ of the Poisson kernel are uniformly bounded on $K\times D$. Therefore \eqref{eq:L1-convergence} gives
\begin{equation}\label{eq:C2-extension-convergence}
 W_j\longrightarrow W
 \qquad\text{in }C^2(K).
\end{equation}
For every $j$, the smooth-domain result \eqref{eq:full-inertia-smooth} gives
\[
 \In D^2W_j=(1,n)
 \qquad\text{in }\R^N_+.
\]
Inertia is not closed under matrix convergence, since eigenvalues may converge to zero. The set of symmetric matrices with at most one positive eigenvalue is closed. Hence \eqref{eq:C2-extension-convergence} yields
\begin{equation}\label{eq:limit-one-positive-upper}
 n_+(D^2W)\le1
 \qquad\text{in }\R^N_+.
\end{equation}
Reflection across $y=0$ conjugates the Hessian by $\diag(I_n,-1)$, so the same inequality holds in the lower half-space. At every $x\in D$, interior elliptic regularity for the fractional Laplacian \cite{Grubb2015} and weak Neumann reflection imply that $W$ is smooth and harmonic in a full neighborhood of $(x,0)$. Letting $y\downarrow0$ gives $n_+(D^2W(x,0))\le1$. Thus
\begin{equation}\label{eq:limit-one-positive-slit}
 n_+(D^2W)\le1
 \qquad\text{throughout }
 \cG_D=\R^N\setminus(D^c\times\{0\}).
\end{equation}

Choose $x_*\in D$. The segment $\{x_*\}\times[-1,1]$ joins the upper and lower half-spaces inside $\cG_D$, so $\cG_D$ is path connected. Apply \cref{cor:GW-propagation} to $q=-W$. Its Hessian has at most one negative eigenvalue by \eqref{eq:limit-one-positive-slit}, while \cref{lem:far-point} provides a nonsingular point. The corollary first gives $\det D^2W\ne0$ throughout $\cG_D$. Since $\tr D^2W=0$ and $n_+(D^2W)\le1$, the nonsingular Hessian has at least one positive eigenvalue and therefore exactly one; the remaining $n$ eigenvalues are negative. Hence
\begin{equation}\label{eq:limit-full-inertia}
 \In D^2W=(1,n)
 \qquad\text{throughout }\cG_D.
\end{equation}
This use of the Gleason--Wolff zero-set theorem is necessary because $C^2_{\mathrm{loc}}$ convergence alone does not exclude zero eigenvalues.

It remains to prove that $W_{yy}(x,0)>0$ for $x\in D$. Put
\[
 Z_j=(W_j)_y,
 \qquad
 Z=W_y
 \qquad(y>0).
\]
By \eqref{eq:Z-positive}, $Z_j>0$ in the upper half-space. The convergence \eqref{eq:C2-extension-convergence} implies $Z_j\to Z$ locally uniformly, and hence $Z\ge0$. The far-field expansion gives $Z(X)\to1$ as $|X|\to\infty$, so $Z$ is not identically zero. The strong maximum principle yields
\begin{equation}\label{eq:limit-Z-positive}
 Z>0
 \qquad\text{in }\R^N_+.
\end{equation}
Fix $x\in D$. On a sufficiently small upper half-ball centered at $(x,0)$, the function $Z$ is harmonic, vanishes on the flat part of the boundary, and is positive inside. The Hopf lemma gives
\begin{equation}\label{eq:limit-Wyy-positive}
 W_{yy}(x,0)=Z_y(x,0)>0.
\end{equation}
The even reflection gives the block decomposition
\begin{equation}\label{eq:limit-bottom-block}
 D^2W(x,0)
 =\begin{pmatrix}
 D^2u(x)&0\\
 0&W_{yy}(x,0)
 \end{pmatrix}.
\end{equation}
By \eqref{eq:limit-full-inertia}, the extended Hessian has one positive and $n$ negative eigenvalues. Since the $yy$ entry in \eqref{eq:limit-bottom-block} is positive, the $x$-Hessian block is negative definite, and therefore
\begin{equation}\label{eq:limit-D2u-strict}
 D^2u(x)<0
 \qquad\text{for every }x\in D.
\end{equation}
Interior elliptic regularity for $(-\Delta)^{1/2}$ \cite{Grubb2015} gives $u\in C^\infty(D)$, so \eqref{eq:limit-D2u-strict} holds classically.

To obtain the line-segment formulation of strict concavity, take distinct $x_0,x_1\in D$ and set
\[
 g(t)=u((1-t)x_0+tx_1),
 \qquad0\le t\le1.
\]
Convexity of $D$ places the entire segment in $D$, and \eqref{eq:limit-D2u-strict} gives
\[
 g''(t)
 =(x_1-x_0)^{\mathsf T}D^2u((1-t)x_0+tx_1)(x_1-x_0)<0
 \qquad(0<t<1).
\]
Thus $g$ is strictly concave on $[0,1]$, which proves the chord statement and completes the proof of \cref{thm:main}.

\appendix

\section[Explicit (r,y)-Hessian determinant for the unit ball]{Explicit computation of the $(r,y)$-Hessian determinant for the unit ball}\label{app:ball-algebra}

We derive \eqref{eq:ball-plane-det}. Set
\[
 A=1+\xi^2,
 \qquad
 B=1-\eta^2,
 \qquad
 \delta_{\mathrm{obl}}=\xi^2+\eta^2,
 \qquad
 K=f_n-\xi f_n'.
\]
The ODE for $f_n$ is
\begin{equation}\label{eq:ball-ode-app}
 Af_n''+n\xi f_n'-nf_n=0.
\end{equation}
Formula \eqref{eq:oblate-derivatives} gives
\[
 \frac{(U_B)_r}{r}=-\frac{\eta K}{\delta_{\mathrm{obl}}}.
\]
Differentiate once more, substitute the inverse-coordinate formulas, and use the ODE \eqref{eq:ball-ode-app} to eliminate $f_n''$. After collecting the common factor $K$, one obtains
\[
 (U_B)_{rr}=\frac{\eta K}{\delta_{\mathrm{obl}}^3}H_n,
 \qquad
 (U_B)_{ry}=\frac{\xi\sqrt{AB}\,K}{\delta_{\mathrm{obl}}^3}J_n,
\]
where
\begin{align*}
 H_n&=-\delta_{\mathrm{obl}}^2
 +n\xi^2B\delta_{\mathrm{obl}}+AB(3\xi^2-\eta^2),\\
 J_n&=(n+3)\xi^2\eta^2+(n-1)\eta^4+3\eta^2-\xi^2.
\end{align*}
Harmonicity gives
\[
 (U_B)_{yy}=-(U_B)_{rr}-(n-1)\frac{(U_B)_r}{r}.
\]
Substituting these expressions into the determinant reduces the calculation to the polynomial identity
\[
 \eta^2H_n\{(n-1)\delta_{\mathrm{obl}}^2-H_n\}
 -AB\xi^2J_n^2
 =-A\delta_{\mathrm{obl}}^3\{1+(n-1)\eta^2\}.
\]
To verify this identity, put $X=\xi^2$ and $Y=\eta^2$, so that $A=1+X$, $B=1-Y$, and $\delta_{\mathrm{obl}}=X+Y$. After substituting the definitions of $H_n$ and $J_n$, the left-hand side is affine in $n$. Its constant coefficient is
\[
 (1+X)(X+Y)^3(Y-1),
\]
and its coefficient of $n$ is
\[
 -Y(1+X)(X+Y)^3.
\]
There is no $n^2$ term. Adding the two contributions gives
\[
 -(1+X)(X+Y)^3\{1+(n-1)Y\},
\]
which is exactly the right-hand side. The identity therefore yields
\[
 (U_B)_{rr}(U_B)_{yy}-(U_B)_{ry}^2
 =-\frac{AK^2\{1+(n-1)\eta^2\}}{\delta_{\mathrm{obl}}^3},
\]
which is \eqref{eq:ball-plane-det}.

\section*{Acknowledgments}
The authors thank Professor Xi-Nan Ma for bringing this problem to their attention. The first author is supported by the National Natural Science Foundation of China [Grant No. 2025YFA1017601]. The second author is  supported by the Program for Young Talents of Basic Research in University of Heilongjiang Province [Grant No. YQJH2025116].  The authors acknowledge the use of AI tools. All mathematical statements and proofs were independently verified by the authors, who take full responsibility for the content of the manuscript.

\end{document}